\documentclass[11pt, oneside]{article}

\usepackage{amsthm}
\usepackage{amsmath}
\usepackage{amsfonts}
\usepackage{amssymb}
\usepackage{fullpage}

\theoremstyle{plain}
\newtheorem{thm}{\textbf{Theorem}}[section]

\newtheorem{lem}[thm]{\textbf{Lemma}}
\newtheorem{cor}[thm]{\textbf{Corollary}}
\newtheorem{pro}[thm]{\textbf{Proposition}}

\newtheorem*{kt}{\textbf{Kneser's Theorem}}

\newtheorem*{kemp}{\textbf{Kemperman Structure Theorem (KST)}}

\newcommand{\Sum}[2]{\underset{#1}{\overset{#2}{\sum}}}

\newcommand{\Z}{\mathbb{Z}}

\newcommand{\db}{d^{\subseteq}}

\newcommand{\be}{\begin{equation}}
\newcommand{\ee}{\end{equation}}
\newcommand{\Summ}[1]{\underset{#1}{\sum}}

\begin{document}

\title{A Step Beyond Kemperman's Structure Theorem}
\author{David J. Grynkiewicz\thanks{Universitat Polit\`{e}cnica de Catalunya,
Barcelona}
\thanks{This Research was supported in part by the
National Science Foundation, as a Math and Physical Sciences
Distinguished International Postdoctoral Research Fellow, under
grant DMS-0502193.}}

\maketitle

\section{Introduction}
\indent\indent Let $G$ be an abelian group, and let $A$ and $B$ be
nonempty subsets of $G$. Their sumset is the set of all pairwise
sums, i.e., $A+B=\{a+b\mid a\in A,\,b\in B\}$. For a set
$\mathcal{S}$ of subsets of $G$, define
$$\db(A,\mathcal{S})=\underset{B\in \mathcal{S}}{\min}\{
\db(A,B)\},$$ where $\db(A,B)=|B\setminus A|$, if $A\subseteq B$,
and $\db(A,B)=\infty$ otherwise. Hence $\db(A,\mathcal{S})$ measures
how far away as a subset the set $A$ is from the sets $B\in
\mathcal{S}$.

It is the central problem of inverse additive theory to describe the
structure of those pairs of subsets $A$ and $B$ with $|A+B|$ small.
Such descriptions often prove useful to other related areas of
mathematics---a notable example being the use of Freiman's Theorem
\cite{Freiman-Theorem-by-Freiman} \cite{natbook}, describing
$A\subseteq \Z$ with $|A+A|<c|A|$, to give a more quantitative proof
of Szemeredi's Theorem \cite{szemeredi} concerning the existence of
(4-term) arithmetic progressions in a subset of positive upper
density \cite{Gowers-szemerdiproof}.

One of the classical results of inverse additive theory was the
complete recursive description given by Kemperman \cite{kst} of the
`critical pairs' in an abelian group, i.e., those finite, nonempty
subsets $A$ and $B$ such that $|A+B|<|A|+|B|$. Among other
applications---including results in graph theory
\cite{ham-kst-vosper-graph-paper} and zero-sum additive theory
\cite{EGZ-hyp}---Kemperman's Structure Theorem (KST), whose
statement we delay until later, yields the descriptions of those
subsets of a locally compact abelian group whose Haar measure of the
sumset fails to satisfy the triangle inequality \cite{kst}
\cite{kneser-analytic-localcompact}. Other applications may also be
found in \cite{serro-smallsumset-survery} \cite{levnewkemp}.

Unfortunately, KST has not perhaps been appreciated or utilized to
its full potential, in part due to the cloud of confusion and
misunderstanding stemming from the perceived complexity of the
theorem's statement. In fact, several papers have been published
with such goals as the simplification of the conclusion of KST
\cite{lev}, the creation of alternative methods for dealing with
critical pairs \cite{ham2-recursive-descrip}
\cite{ham3-vosper-prop}, and the clarification of the use of KST in
practice \cite{Quasi-periodic-interpret-KST} \cite{levnewkemp}.

Concerning the structure of $A$ and $B$ when $A+B$ is small, few
precise results besides KST are known for an arbitrary abelian
group. In the case $G=\Z$ with $|A|\geq |B|$, then (currently, a few
technical restrictions are also needed)
\begin{equation}\label{intro_chat-1}|A+B|=|A|+|B|+r\leq
|A|+2|B|-4+\epsilon\end{equation} implies that
$\db(A,\mathcal{AP}_d),\,\db(B,\mathcal{AP}_d)\leq r+1$ for some
$d$, where $\mathcal{AP}_d$ is the set of arithmetic progressions
with difference $d$, and $\epsilon$ equals $0$ or $1$, depending on
a structural condition between $A$  and $B$ \cite{natbook}. Thus $A$
and $B$ must be large subsets of arithmetic progressions with the
same difference.

In the case $G=\Z/p\Z$, where $p$ is prime, then the critical
pairs---with $|A|\geq |B|>1$ and $|A+B|\leq p-2$ (to avoid three
very special degenerate examples)---consist of arithmetic
progressions with the same difference \cite{Vosper-thm}
\cite{vosper-addend}. Under some additional restrictions on the
cardinality of $A$, and assuming $A=B$, then a result of Freiman
shows $|A+A|\leq 2|A|+r\leq 2.4|A|$ implies that
$\db(A,\mathcal{AP}_d)\leq r+1$, for some $d$
\cite{Freiman-vosper-I} \cite{Freiman-vosper-II} \cite{natbook}; in
other words, the above result from $\Z$ holds for $A=B$, by imposing
some moderate conditions, in $\Z/p\Z$ as well. The same result is
also known for $\Z/p\Z$ under the more general assumption of
(\ref{intro_chat-1}), provided extremely severe conditions are
imposed on the cardinalities of $A$ and $B$
\cite{Rectification-paper}. The extent to which the result holds in
$\Z/p\Z$ without unnecessary assumptions on the cardinalities is
still quite open. Little is known beyond the case $|A+B|=|A|+|B|$,
for which Hamidoune and R{\o}dseth established
$\db(A,\mathcal{AP}_d),\,\db(B,\mathcal{AP}_d)\leq 1$, with only the
assumption $|A+B|\leq p-4$ and the removal of $\epsilon$ from
(\ref{intro_chat-1}) \cite{ham-rodseth}; and the case
$|A+B|=|A|+|B|+1$, for which
$\db(A,\mathcal{AP}_d),\,\db(B,\mathcal{AP}_d)\leq 2$ was shown by
Hamidoune, Serra, and Zemor, under similar assumptions with $p>51$
\cite{vosp+2}. Concerning more general abelian groups, for
$A\subseteq \Z/mZ$ with $|A+A|<2.04 |A|$, Deshouillers and Freiman
obtained a rough description of $A$ involving computed large
constants \cite{step-beyond-kneser}.

In this paper, we move one step beyond KST by completing the
description of all subsets $A$ and $B$ that exactly achieve (rather
than fail to achieve) the triangle inequality, namely for which
$|A+B|=|A|+|B|$. Our main result is Theorem \ref{KST_Step_Beyond}
(whose statement we also delay until further notation and concepts
have been developed), which shows that with a few noted
exceptions---all but one in the same vein as the original recursive
description of KST---then such $A$ and $B$ must be large subsets of
a critical pair; more specifically, there must exist $A'\supseteq A$
and $B'\supseteq B$ such that $|A'+B'|=|A'|+|B'|-1$, and which
contain $A$ and $B$ each with at most one hole, i.e., $|A'\setminus
A|\leq 1$ and $|B'\setminus B|\leq 1$. Thus in the case $G=\Z/p\Z$,
with $p$ prime, Theorem \ref{KST_Step_Beyond} generalizes the prime
case completed by Hamidoune and R{\o}dseth, and is the corresponding
composite extension of KST. Theorem 4.1 and KST will also yield
necessary and sufficient conditions for $|A+B|=|A|+|B|$.

We should remark that Hamidoune, Serra and Zemor very recently
established a particular case of Theorem 4.1, under a series of
added assumptions, including that $\gcd(|G|,6)=1$, that $A$ be a
generating subset, that the order of every element of $A\setminus 0$
be at least $|A|+1$, and that a few smaller technical assumptions
also hold \cite{ham-ser-chowla}---the hypotheses needed for their
result, particularly the assumption on the order of elements,
parallel other first-attempt generalizations of additive results,
from the prime order case to the more general abelian group setting
(see \cite{chowla} \cite{EHC-extremal-cases} \cite{EHC-chowla-karol}
for other such examples).


\section{Preliminaries}

\indent \indent We will make heavy use of the interpretation of KST
given in \cite{Quasi-periodic-interpret-KST} (or in
\cite{PhD-Dissertation}, where the explanations are slightly
extended, including the expansion of a minor omission in comment
(c.12) of \cite{Quasi-periodic-interpret-KST}), some of which may
also be found in \cite{levnewkemp}. First we begin by describing
many of the important definitions and notation that we will use.

Let $G$ be an abelian group. A subset $A\subseteq G$ is
\emph{$H_a$-periodic} if $A$ is a union of $H_a$-cosets, with $H_a$
a subgroup (referred to as the \emph{period}). Note that every set
is $H_a$-periodic with $H_a$ the trivial group. If $A$ is
$H_a$-periodic with $H_a$ a nontrivial subgroup, then $A$ is
\emph{periodic}, and otherwise $A$ is \emph{aperiodic}. Note that
$A$ being $H_a$-periodic is equivalent to $A+H_a=A$. Hence if $A$ is
$H_a$-periodic, then so is $A+B$. An $H_a$-hole in $A$ is an element
of $(A+H_a)\setminus A$, and when clear, $H_a$ will be dropped from
the notation. A \emph{punctured periodic set}, is a set $A$ such
that $A\cup\{\gamma\}$ is periodic for some $\gamma\notin A$, i.e.,
$A$ contains exactly one $H_a$-hole for some nontrivial $H_a$. We
remark that a punctured periodic set cannot be periodic (as for
instance shown in \cite{Quasi-periodic-interpret-KST}
\cite{PhD-Dissertation}). We use $\phi_a:G\rightarrow G/H_a$ to
denote the natural homomorphism. Note that if $A$ is maximally
$H_a$-periodic (meaning $H_a$ is the maximal subgroup for which $A$
is $H_a$-periodic, sometimes called the stabilizer), then
$\phi_a(A)$ is aperiodic. One of the foundational results of
additive theory is the following result of Kneser \cite{kt}
\cite{kt-asymptotic} \cite{kst} \cite{natbook}.
\begin{kt}
Let $G$ be an abelian group, and let $A_1,A_2,\ldots,A_n$ be a
collection of finite, nonempty subsets of $G$. If
$\underset{i=1}{\overset{n}{\sum }}A_i$ is maximally $H_a$-periodic,
then
$$\left|\Sum{i=1}{n} \phi_a(A_i)\right|\geq \Sum{i=1}{n}
|\phi_a(A_i)|-n+1.$$
\end{kt}

Note that if $A+B$ is maximally $H_a$-periodic and
$\rho=|A+H_a|-|A|+|B+H_a|-|B|=\db(A,A+H_a)+\db(B,B+H_a)$ is the
number of holes in $A$ and $B$, then Kneser's Theorem implies (by
multiplying all terms by $|H_a|$) that $|A+B|\geq
|A|+|B|-|H_a|+\rho$. Consequently, if either $A$ or $B$ contains a
unique element from some $H_a$-coset, then $|A+B|\geq |A|+|B|-1$.
Also, if $|A+B|\leq |A|+|B|-1$, then equality holds in the bound
from Kneser's Theorem (else $|A+B|\geq
|H_a|(|\phi_a(A)|+|\phi_a(B)|)\geq |A|+|B|$).

Given $a_i\in A$ and a subgroup $H_a$, we use $A_{a_i,H_a}$ to
denote $(a_i+H_a)\cap A$, with the $H_a$ dropped from the notation
when clear. If $A_{a_i}\neq a_i+H_a$, then $A_{a_i}$ is a
\emph{partially filled $H_a$-coset}. An \emph{$H_a$-decomposition}
of $A$ is a partition $A_{a_1}\cup\ldots\cup A_{a_l}$ of $A$ with
$a_i\in A$. The compliment of $A$ is denoted $\overline{A}$, and we
use $\langle A\rangle$ to denote the subgroup generated by $A$,
which, when $0\in A$, is the smallest subgroup $H_a$ such that
$|\phi_a(A)|=1$.

A \emph{quasi-periodic decomposition} of $A$ with
\emph{quasi-period} $H_a$, where $H_a$ is a nontrivial subgroup, is
a partition $A_1\cup A_0$ of $A$ into two disjoint (each possibly
empty) subsets such that $A_1$ is $H_a$-periodic or empty, and $A_0$
is a subset of an $H_a$-coset. Note every set has a quasi-periodic
decomposition with $H_a=G$ and $A_1=\emptyset$. A set $A\subseteq G$
is \emph{quasi-periodic} if $A$ has a quasi-periodic decomposition
$A=A_1\cup A_0$ with $A_1$ nonempty. Given a quasi-periodic
decomposition $A_1\cup A_0$ with quasi-period $H_a$, then $A_1$ is
the \emph{periodic part} of the decomposition, and $A_0$ is the
\emph{aperiodic part} (although it may be periodic if $A$ is
periodic). Such a decomposition is \emph{reduced} if $A_0$ is not
quasi-periodic. Note that if $A$ is finite and has a quasi-periodic
decomposition $A_1\cup A_0$ with quasi-period $H$, then $A$ has a
reduced quasi-periodic decomposition $A'_1\cup A'_0$ with
quasi-period $H'\leq H$ and $A'_0\subseteq A_0$.

An \emph{arithmetic progression} with difference $d\in G\setminus 0$
and length $l\in\mathbb{N}$ is a set of the form
$\{a_0,a_0+d,\ldots,a_0+(l-1)d\}$, where $a_0\in G$. The terms $a_0$
and $a_0+(l-1)d$ are the \emph{end terms} in the progression, with
$a_0$ the first term and $a_0+(l-1)d$ the last term. Note that an
arithmetic progression with difference $d$ is also an arithmetic
difference with difference $-d$, with the first and last terms
interchanged. A \emph{$d$-component} of a set $A$ is a maximal
arithmetic progression with difference $d$ contained in $A$. We use
$c_d(A)$ to denote the number of aperiodic $d$-components of $A$.
Note that if a $d$-component is periodic, then it must be a $\langle
d\rangle$-coset. Hence $c_d(A)=|A+\{0,d\}|-|A|$. A
\emph{quasi-progression} with difference $d$ is a set $P$ with a
quasi-periodic decomposition $P=P_1\cup P_0$ with quasi-period
$\langle d\rangle$, such that $P_0$ is an arithmetic progression
with difference $d$. We use $l_d(A)$ to denote the cardinality of
the minimal quasi-progression $P$ with difference $d$ that contains
$A$, and $h_d(A)=l_d(A)-|A|=|P\setminus A|$ counts the number of
holes in $A$ with respect to such a minimal quasi-progression $P$.

Assuming $0\in A\cap B$, for $i\geq 0$ we define the set $N_i(A,B)$
by
\begin{itemize}
    \item $N_0(A,B)=A,$
    \item $N_i(A,B)=(A+iB)\setminus (A+(i-1)B) \hbox{ for }
i\geq 1,$
  \end{itemize} where $0B=\{0\}$ and $iB=\underset{i}{\underbrace{B+\ldots+B}}$
for $i\geq 1$. If the sets $A$ and $B$ are clear, they will often be
dropped from the notation. For $x\in G$, let $$\nu_x(A,B)=(x-A)\cap
B.$$ Hence $|\nu_x(A,B)|=|\nu_x(B,A)|$ is the number of
representations of $x=a+b$ with $a\in A$ and $b\in B$. For
$U\subseteq B$, $i\geq 1$, and $N_i$ nonempty, we use $N^U_i$ to
denote the set of all elements $x\in N_i$ such that
$\nu_x(A+(i-1)B,B)=U$, and we define $N^{\leq
U}_i=\underset{V\subseteq U}{\bigcup}N_i^V.$ Hence
$A+(i-1)B+(B\setminus U)=(A+iB)\setminus N_i^{\leq U}$. In
particular, $|N_1^b(A,B)|$ is the number of $a\in A$ with $a+b$ a
unique expression element in $A+B$. The sets $N_i$ were first
introduced in \cite{2-component-prime-case-oriol-zemor} in
connection with small sumsets in $\Z/p\Z$, with $p$ prime, and have
since shown themselves to be quite useful \cite{ham-ser-chowla}.

Let $\mathcal{P}$ be the set of periodic subsets of $G$, let
$\mathcal{QP}$ be the set of quasi-periodic subsets of $G$, let
$\mathcal{QP}_H$ be the set of quasi-periodic subsets of $G$ with
quasi-period $H$, let $\mathcal{P}_H$ be the set of $H$-periodic
subsets of $G$, let $\mathcal{AP}$ be the set of arithmetic
progressions, and let $\mathcal{AP}_d$ be the set of arithmetic
progressions with difference $d$. Note that $\mathcal{P}\subseteq
\mathcal{QP}$.

We will need the following basic proposition \cite{natbook}
\cite{kst}. Note that if $G$ is finite and $|A|+|B|\geq |G|+r$, then
Proposition \ref{mult_result} implies $|\nu_g(A,B)|\geq r$ for every
$g\in G$.

\begin{pro}\label{mult_result} Let $A$
and $B$ be nonempty, finite subsets of an abelian group $G$, and let
$r\in \Z$.

(i) If $G$ is finite, and $|A|+|B|\geq |G|+1$, then $A+B=G$.

(ii) If $|A+B|<|A|+|B|-r$, then $|\nu_g(A,B)|> r$ for every $g\in
A+B$.
\end{pro}

We can now give the statement of KST \cite{kst}.

\begin{kemp} \label{KST} Let $A$ and $B$ be finite, nonempty subsets of an
abelian group $G$. Then $|A+B|=|A|+|B|-1,$ and, moreover, if $A+B$
is periodic then $|\nu_c(A,B)|=1$ for some $c$, if and only if there
exist quasi-periodic decompositions $A=A_1\cup A_0$ and $B=B_1\cup
B_0$ with common
quasi-period $H_a$, and $A_0$ and $B_0$ nonempty, such that:\\
$(i)$ $|\nu_c(\phi_a(A),\phi_a(B))|=1$, where  $c=
\phi_a(A_0)+\phi_a(B_0)$ \\
$(ii)$ $|\phi_a(A)+\phi_a(B)|=|\phi_a(A)|+|\phi_a(B)|-1$, \\
$(iii)$ $|N_1^b(A,B)|=|N_1^a(B,A)|=0$ for all $a\in A_1$ and $b\in
B_1$, and\\
$(iv)$ the pair $(A_0,B_0)$ is of one of the following types $($each
of which imply $|A_0+B_0|=|A_0|+|B_0|-1)$:

$(I)$ $|A_0|=1$ or $|B_0|=1$;

$(II)$ $A_0$ and $B_0$ are arithmetic progressions with common
difference $d$, where the order of $d$ is at least $|A_0|+|B_0|-1$,
and $|A_0|\geq 2$, $|B_0|\geq 2$ $($hence, $A_0+B_0$ is an
arithmetic progression with difference $d$, while
$|\nu_{c}(A_0,B_0)|=1$ for exactly two $c\in A_0+B_0)$;

$(III)$ $|A_0|+|B_0|=|H_a|+1$, and precisely one element $g_0$
satisfies $|\nu_{g_0}(A_0,B_0)|=1$ $($hence, $B_0$ has the form
$B_0=(g_0-((g_1+H_a)\cap\overline{A_0}))\cup\{g_0-g_1\}$, where
$g_1\in A_0)$;

$(IV)$ $A_0$ is aperiodic, $B_0$ is of the form
$B_0=g_0-((g_1+H_a)\cap\overline{A_0})$, with $g_1\in A_0$ $($hence,
$A_0+B_0=(g_0+H_a)\setminus g_0)$, and $|\nu_{c}(A_0,B_0)|\neq 1$
for all $c$.
\end{kemp}

The condition (iii) was not shown in Kemperman's original paper, but
can be derived from KST as shown in
\cite{Quasi-periodic-interpret-KST} \cite{PhD-Dissertation}.
Conditions (i) and (ii) imply that the pair $(\phi_a(A),\phi_a(B))$
satisfies the hypothesis of KST. Hence repeated application of KST
modulo the quasi-period yields a complete recursive description of
the pair $(A,B)$. This gives rise to a chain of subgroups
$$0=H_{a_0}<H_a=H_{a_1} <H_{a_2}<\ldots <H_{a_n}=G,$$ where
$H_{a_{i}}/H_{a_{i-1}}$ is the quasi-period given by KST after $i$
iterations. Conditions (i) and (iii) ensure `proper alignment'
during this recursive process, namely that $\phi_a(A_0)$ and
$\phi_a(B_0)$ will always be contained in the aperiodic part of the
mod $H_a$ quasi-periodic decomposition given by KST. Consequently,
KST induces partitions (allowing empty parts) $A=A_n\cup
A_{n-1}\cup\ldots\cup A_1\cup A_0$ and $B=B_n\cup
B_{n-1}\cup\ldots\cup B_1\cup B_0$, such that $A_n=B_n=\emptyset$,
$A_0$ and $B_0$ are nonempty, and
$$\phi_{a_{j-1}}(A)=\phi_{a_{j-1}}\left(A_n\cup \ldots \cup A_j\right)\bigcup
\phi_{a_{j-1}}\left(A_{j-1}\cup \ldots \cup A_0\right),$$ and
$$\phi_{a_{j-1}}(B)=\phi_{a_{j-1}}\left(B_n\cup \ldots \cup B_j\right)\bigcup
\phi_{a_{j-1}}\left(B_{j-1}\cup \ldots \cup B_0\right)$$ are the
quasi-periodic decompositions with quasi-period
$H_{a_j}/H_{a_{j-1}}$ given by KST after $j$ applications, while
$$A=\left(A_n\cup \ldots \cup A_j\right)\bigcup \left(A_{j-1}\cup
\ldots \cup A_0\right),$$ and $$B=\left(B_n\cup \ldots \cup
B_j\right)\bigcup \left(B_{j-1}\cup \ldots \cup B_0\right)$$ are
also quasi-period decompositions of $A$ and $B$ with common
quasi-period $H_{a_j}$.

At first it might appear that the critical pairs with $A+B$ periodic
with maximal period $H_a$, including the cases when $|A+B|<|A|+|B|$,
are not fully covered by KST. However these cases are easily reduced
to the cases covered by KST. This is because in view of Kneser's
Theorem it follows that $|\phi_a(A+B)|=|\phi_a(A)|+|\phi_a(B)|-1$,
with $\phi_a(A+B)$ aperiodic, and $A+B$ containing exactly
$|A+B|+|H_a|-|A|-|B|$ holes. Thus KST is used to describe the pair
$(\phi_a(A),\phi_a(B))$, and then $A$ and $B$ are obtained from
$A+H_a$ and $B+H_a$ by deleting $|A+B|+|H_a|-|A|-|B|\leq |H_a|-1$
elements. In view of Proposition \ref{mult_result}, these deleted
elements can be placed anywhere in the sets $H_a+A$ and $H_a+B$,
with the resulting sets being critical. Thus there is little to say
about the location of these holes.

It is natural to wonder if a pair $(A,B)$ could have more than one
quasi-periodic decomposition that satisfies KST, in other words, how
unique is the representation given by KST. This question is
addressed in \cite{Quasi-periodic-interpret-KST}
\cite{PhD-Dissertation}. We provide a short summary. The type of a
critical pair $(A,B)$ is unique, and not dependent on the choice of
quasi-periodic decompositions that satisfy KST. If the pair $(A,B)$
has type (II), then there is a unique pair of quasi-periodic
decompositions that satisfy KST. The same is true (in view of the
added condition (iii)) for type (I). In both cases, the quasi-period
$H_a$ may not be unique, but can always be taken to be maximal,
subject to the conditions of KST. There is slightly more freedom for
types (III) and (IV), but in both cases there is a natural canonical
choice. For type (III), the quasi-period may always be taken to be
the maximal period of $A+B$. This ensures that type (III) will not
occur twice in a row when recursively iterating KST. Type (IV) can
only occur in the first iteration of KST (since a type (IV) pair has
no unique expression element), and implies that $A+B$ is a punctured
$H_a$-periodic set with $|H_a|\geq 6$. These conditions imply that
there is a unique $\gamma\in \overline{A+B}$ such that $A+B\cup
\{\gamma\}$ is periodic, and $H_a$ may always be taken to be the
maximal period of $A+B\cup \{\gamma\}$. This ensures that type (III)
not follow type (IV) when iterating KST. Hence for all types the
quasi-period $H_a$ given by KST may always be taken to be maximal,
and the resulting quasi-period decompositions that satisfy KST will
be referred to as the \emph{Kemperman decompositions}.

Regarding the sequence of possible types, there is some restriction
on when type (I) can occur twice in a row, assuming $H_a$ chosen
maximally. Namely, if $A=A_1\cup A_0$ and $B=B_1\cup B_0$ are the
Kemperman decompositions of $A$ and $B$ with common quasi-period
$H_a$, and if $\phi_a(A)=\phi_a(A'_1)\cup \phi_a(A'_0)$ and
$\phi_a(B)=\phi_a(B'_1)\cup \phi_a(B'_0)$ are the Kemperman
decompositions of $\phi_a(A)$ and $\phi_a(B)$, then we cannot have
both $(A,B)$ of type (I) with $|A_0|=1$, and also
$(\phi_a(A),\phi_a(B))$ of type (I) with $|\phi_a(A'_0)|=1$---this
is the main step in the proof of Proposition 2.2 in
\cite{Quasi-periodic-interpret-KST}, and likewise for Proposition
5.3 in \cite{PhD-Dissertation}. Finally, it should also be observed
that if $A$ is not quasi-periodic and $\langle -\alpha+A\rangle=G$,
with $\alpha\in A$, then $H_a=G$.

We will also need the following two simple propositions that, like
the proofs given here, are minor variations on two results from
\cite{2-component-prime-case-oriol-zemor}.

\begin{pro}\label{surjectivity_transfer}
Let $G$ be an abelian group, let $X,\,Y\subseteq G$ be finite and
nonempty with $0\in X\cap Y$, and let $i\geq 1$. If
$|\nu_z(X,Y)|\geq t$ for every $z\in X+Y$, then
$|\nu_z(X+(i-1)Y,Y)|\geq t$ for every $z\in X+iY$.
\end{pro}

\begin{proof}
It suffices to prove the case $i=2$, as the other cases follow by
repeated application. Let $y,y'\in Y$ and $x\in X$. Since
$|\nu_z(X,Y)|\geq t$ for $z\in X+Y$, let  $\{x_i+y_i\}_{i=1}^t$ be
distinct representations of $x+y$, with $x_i\in X$ and $y_i\in Y$.
Hence $\{(x_i+y')+y_i\}_{i=1}^t$ are distinct representations of
$x+y+y'\in X+2Y$ with $x_i+y'\in X+Y$ and $y_i\in Y$, whence
$|\nu_z(X+Y,Y)|\geq t$ for every $z\in X+2Y$.
\end{proof}

\begin{pro}\label{magic_N_i_lemma}
Let $G$ be an abelian group, let $X,\,Y\subseteq G$ be finite and
nonempty with $0\in X\cap Y$, and let $i\geq 1$. If $U\subseteq Y$,
then $N^U_{i+1}(X,Y)-U\subseteq N_i^{\leq U}(X,Y).$
\end{pro}

\begin{proof}
Let $x\in N_{i+1}^U$. Hence by definition $U$ is the maximal subset
of $Y$ such that $x-U\subseteq X+iY$. Furthermore, $x-U\subseteq
N_i$, since otherwise $x-u=y$ for some $u\in U\subseteq Y$ and $y\in
X+(i-1)Y$, whence $x=u+y\in X+iY$, contradicting that $x\in
N_{i+1}^U\subseteq N_{i+1}$. If $x-U\nsubseteq N_i^{\leq U}$, then
for some $u\in U$, it follows that $x-u=z\in N_{i}$, with $z=u'+x'$
for some $u'\in Y\setminus U$ and $x'\in X+(i-1)Y$. Thus
$x-u'=u+x'\in X+iY$. Hence, since $u'\notin U$, this contradicts the
maximality of $U$.
\end{proof}

We conclude the section with one last important concept.  Given a
pair of subsets $A$ and $B$ of an abelian group $G$, we say that $A$
is \emph{non-extendible} (with respect to $B$), if

\begin{equation}\label{A_nonex}A\cup\{a_0\}+B\neq A+B, \mbox{ for all }a_0\in
\overline{A}.\end{equation} We say $A$ is \emph{extendible}
otherwise. We will call the pair $(A,B)$ \emph{non-extendible} if
both $A$ and $B$ are non-extendible (with respect to each other),
and \emph{extendible} otherwise. Note that in general

\begin{equation}\label{A_half_dual}-B+\overline{A+B}\subseteq
\overline{A},\end{equation} since otherwise $-b+c_0=a$ for some
$b\in B$, $c_0\in \overline{A+B}$, and $a\in A$, implying
$c_0=a+b\in A+B$, contradicting $c_0\in \overline{A+B}$. However,
$A$ being non-extendible is in fact equivalent to equality holding.

\begin{pro} \label{AB_dual} For subsets
$A$ and $B$ of an abelian group, $A$ is non-extendible if and only
if
\begin{equation}\label{A_dual}-B+\overline{A+B}=
\overline{A}.\end{equation} In particular, the pair $(A,B)$ is
non-extendible if and only if both (\ref{A_dual}) and
\begin{equation}\label{B_dual}-A+\overline{A+B}= \overline{B},\end{equation}
hold, and both the pairs $(-A,\overline{A+B})$ and
$(-B,\overline{A+B})$ are also non-extendible.
\end{pro}

\begin{proof}
Suppose $A$ in non-extendible. Let $a_0\in \overline{A}$. In view of
(\ref{A_nonex}), it follows that there exists $b\in B$ and $c_0\in
\overline{A+B}$ such that $a_0+b=c_0$, whence $a_0=-b+c_0\in
-B+\overline{A+B}$. Thus $\overline{A}\subseteq-B+\overline{A+B}$,
and equality follows in view of (\ref{A_half_dual}).

On the other hand, if we suppose (\ref{A_dual}), then each $a_0\in
\overline{A}$ can be written as $a_0=-b+c_0$, with $b\in B$ and
$c_0\in \overline{A+B}$, implying $a_0+b=c_0$ is an element of
$A\cup\{a_0\}+B$ not contained in $A+B$, whence (\ref{A_nonex})
follows.

If $(A,B)$ is non-extendible, and $(-A,\overline{A+B})$ is
extendible, then in view of (\ref{A_dual}), (\ref{B_dual}), and the
first part of the proposition, it would follows that either $A+B\neq
A+B$ or $-\overline{A+B}+B\neq-\overline{A}$. The former is clearly
a contradiction, while in view of (\ref{A_dual}), the later
contradicts the non-extendibility of $A$. The same argument applied
to the pair $(-B,\overline{A+B})$ shows it to likewise be
non-extendible, completing the proof.
\end{proof}

When $G$ is finite, it is important to note that if
$|A+B|=|A|+|B|+r$ with $B$ non-extendible, then proposition
\ref{AB_dual} implies that
$|-A+\overline{A+B}|=|-A|+|\overline{A+B}|+r$ as well. Thus a
non-extendible pair $(A,B)$ with $|A+B|=|A|+|B|+r$ is part of a
triple of non-extendible pairs, all having their sumsets with
cardinality exactly $r$ more than the bound given by the triangle
inequality. The ideas behind Proposition \ref{AB_dual} trace their
roots back to Vosper \cite{Vosper-thm}, and can also be found in the
isoperimetric method (since $k$-fragments are non-extendible)
\cite{ham2-recursive-descrip} \cite{ham3-vosper-prop}.

\section{$A+B$ Versus $A-B$}

\indent \indent For the proof of our main result, we will need the
following basic theorem, which can be viewed as generalization of
the bound for Sidon sets. Indeed, if $A=B$, $|T|=1$ and $k=1$ (the
conditions for $A$ to be a Sidon set, given from the difference set
point of view), then the familiar bound $|A+A|\geq
\frac{|A|(|A|+1)}{2}$ follows from Theorem
\ref{AB_large_from_A-B}(iii) by noting that $x=|A|$. The bound in
Theorem \ref{AB_large_from_A-B}(i) is a general approximation, whose
second half is without constants depending on divisibility, while
Theorem \ref{AB_large_from_A-B}(ii) gives a non-implicit bound for
$|T|$, and Theorem \ref{AB_large_from_A-B}(iii) gives a non-implicit
bound for $|A+B|$.

Theorem \ref{AB_large_from_A-B} will allow us to conclude $|A+B|$ is
large provided $|\nu_x(A,-B)|$ is small for most $x\in A-B$. This
will be important as the proof of our main result uses (a variation
of) the Dyson $e$-transform of the pair $(A,B)$. Namely for $e\in
A-B$, where $|A|\geq |B|$, we define $B(e)=(e+B)\cap A$ and
$A(e)=(e+B)\cup A$. Then $|A(e)|+|B(e)|=|A|+|B|$ while
$A(e)+B(e)\subseteq e+A+B$. Our main strategy will be to apply
induction to the pair $(A(e),B(e))$. However, we will encounter
problems with this strategy if $|B(e)|$ is small for all $e$ such
that $e+B\nsubseteq A$. In these cases, we will use Theorem
\ref{AB_large_from_A-B} to show that $|A+B|$ is large instead.

\begin{thm}\label{AB_large_from_A-B} Let $A$, $B$ and $T$ be finite subsets of an
abelian group $G$, with $|A|\geq |B|>k\geq 1$ and $|A|\geq |T|$. Let
$\delta$ be the integer such that $|B|(|A|-|T|)\equiv \delta\mod k$,
with $0\leq \delta<k$, let
$M=|T||B|(|B|-k)+(k-1)|A||B|-\delta(k-\delta)$, let $x$ be the
integer, $1\leq x\leq |A||B|$, such that $M+x\equiv 0 \mod |A||B|$,
and let $\delta_0$ be the integer such that $|A||B|+\delta_0\equiv 0
\mod |A+B|$, with $0\leq \delta_0<|A+B|$. If $|\nu_x(A,-B)|\leq k$
for all $x\in G\setminus T$, then the following bounds
hold:\begin{description}
                                                                 \item[(i)]
$|A+B|\geq
\frac{|A|^2|B|^2}{M+|A||B|}\geq\frac{|A|^2|B|}{|T|(|B|-k)+k|A|},$
                                                                 \item[(ii)]
$|T|\geq \frac{|A|^2|B|^2-\delta_0^2-|A+B|\left(k|A||B|-\delta_0-
\delta(k-\delta)\right)}{|A+B||B|(|B|-k)}\geq\frac{|A|^2|B|^2-\delta_0^2-|A+B|\left(k|A||B|-\delta_0\right)}{|A+B||B|(|B|-k)}\geq
|A|\frac{|A||B|-k|A+B|}{|A+B|(|B|-k)},$
                                                                 \item[(iii)]
$|A+B|\geq
\frac{2|A||B|}{\left\lfloor\frac{M+2|A||B|}{|A||B|}\right\rfloor}-
\frac{M}{\left\lfloor\frac{M+|A||B|}{|A||B|}\right\rfloor
\left\lfloor\frac{M+2|A||B|}{|A||B|}\right\rfloor}=
\frac{|A|^2|B|^2(M+2x)}{(M+|A||B|+x)(M+x)}.$
                                                               \end{description}
\end{thm}

We first give some basic notions from Graph Theory, in order to put
the ideas of the proof of Theorem \ref{AB_large_from_A-B} in broader
context. For a graph $\Gamma$, we use $\overline{\Gamma}$ to denote
the complement of $\Gamma$, i.e. the graph on the same vertex set
with the edges being the non-edges of $\Gamma$. Also, $V(\Gamma)$
denotes the vertex set, and $E(\Gamma)$ the edge set. The line graph
of $\Gamma$, denoted $L(\Gamma)$, is the graph whose vertices are
the edges of $\Gamma$, with two vertices being adjacent if the
corresponding edges in $G$ share a common vertex. Given a pair of
nonempty subsets $(A,B)$ of an abelian group $G$, we can define a
graph $\mathcal{M}(A,B)$ whose vertices are the ordered pairs from
$A\times B$, with $\{(a,b),(a',b')\}$ an edge precisely when
$a+b=a'+b'$. Hence $\mathcal{M}(A,B)$ is a vertex disjoint union of
cliques $Cl_{x_i}$, with each clique $Cl_{x_i}$ corresponding to an
element $x_i\in A+B$ such that $|\nu_{x_i}(A,B)|=|V(Cl_{x_i})|=m_i$.
Alternatively, $\mathcal{M}(A,B)$ is the complement of the complete
multipartite graph corresponding to the sequence $m_1,\ldots,m_c$,
where $c=|A+B|$, i.e.,
$\mathcal{M}(A,B)=\overline{K_{m_1,m_2,\ldots,m_{c}}}$. Letting
$n_i$ be the number of cliques $Cl_x$ with $|V(Cl_x)|=i$, we note
that
$$|A+B|=\sum_{i\geq 1}n_i,$$
\begin{equation}\label{edge_count}|E(\mathcal{M}(A,B))|=\sum_{i\geq 1}n_i\binom{i}{2}.\end{equation}
The sumset and difference set of $A$ and $B$ are related via these
graphs as follows.

\begin{pro} \label{Graph_Bijection}If $A$ and $B$ are finite, nonempty subsets of an
abelian group $G$, then the map
$\varphi:E(\mathcal{M}(A,B))\rightarrow E(\mathcal{M}(A,-B))$,
defined by $\{(a,b),(a',b')\}\mapsto \{(a,-b'),(a',-b)\}$, is a
bijection that maps adjacent edges in $L(\mathcal{M}(A,B))$ to
distinct components of $\mathcal{M}(A,-B)$. In particular, $\varphi$
embeds
$L(\mathcal{M}(A,B))\hookrightarrow\overline{L(\mathcal{M}(A,-B)}$,
and $|E(\mathcal{M}(A,B))|=|E(\mathcal{M}(A,-B))|$.
\end{pro}

\begin{proof}
Since $a+b=a'+b'$ implies $a-b'=a'-b$, it follows that map $\varphi$
is well defined. Noting that $\varphi$ is its own inverse, it
follows that $\varphi$ is a bijection. If $a_1+b_1=a_2+b_2=a_3+b_3$,
with $a_i\in A$ distinct and $b_i\in B$ distinct, then we note that
$a_1-b_2=a_2-b_1=a_2-b_3=a_3-b_2$ would contradict that $b_1$ and
$b_3$ are distinct. Hence the adjacent edges
$\{(a_1,b_1),(a_2,b_2)\}$ and $\{(a_2,b_2),(a_3,b_3)\}$ are mapped
to distinct components in $\mathcal{M}(A,-B)$, and thus cannot be
adjacent in $L(\mathcal{M}(A,-B))$.
\end{proof}

It is important to note, by means of a simple extremal argument or
discrete derivative, that once $|A+B|$ is fixed, then
(\ref{edge_count}) is minimized by taking all cliques of as near
equal a size as possible. Likewise, (\ref{edge_count}) is maximized
in just the opposite way, by taking as many cliques of largest
allowed size as possible, followed by as many cliques of the next
largest allowed size as possible, \ldots, etc.

For the proof of Theorem \ref{AB_large_from_A-B}, which we now
proceed with, we will need only the equality
$|E(\mathcal{M}(A,B))|=|E(\mathcal{M}(A,-B))|$.

\begin{proof} Let $|A|=a$, $|B|=b$ and $|T|=t$.
Since $|\nu_x(A,-B)|\leq k$ for all $x\in G\setminus T$, since
$a\geq b$, and since $a\geq t$, it follows that
\begin{equation}\label{Differense_set_lemma_LHS}|E(\mathcal{M}(A,-B))|\leq
t\binom{b}{2}+\frac{b(a-t)-\delta}{k}\binom{k}{2}+\binom{\delta}{2}=
\frac{tb^2-tkb+(k-1)ab-\delta(k-\delta)}{2}=\frac{M}{2}.\end{equation}
Thus $M\geq 0$, and $M$ is even. Let
$l=\left\lceil\frac{ab}{|A+B|}\right\rceil-1$, and let $c=|A+B|$.
Note
\begin{equation}\label{c_relates_to_l}c=|A+B|=\frac{ab+\delta_0}{l+1},\end{equation} where
$ab+\delta_0\equiv 0\mod c$, and $0\leq \delta_0<c$. Since
$|E(\mathcal{M}(A,B))|$ will be minimized when $l\leq
|\nu_x(A,B)|\leq l+1$, it follows in view of (\ref{c_relates_to_l})
that
\begin{equation}\label{Difference_set_lemma_RHS}|E(\mathcal{M}(A,B))|\geq
c\binom{l}{2}+(ab-cl)l=\frac{l(2ab-c(l+1))}{2}=\frac{l(ab-\delta_0)}{2}=
(\frac{ab+\delta_0}{c}-1)\frac{(ab-\delta_0)}{2}.
\end{equation}

Considering the above bound as a function of $\delta_0$, we note
that it is quadratic in $\delta_0$ with negative lead coefficient.
Thus this quantity will be minimized at a boundary value for
$\delta_0$. Hence comparing the bound evaluated at $\delta_0=0$ and
$\delta_0=c-1$, it follows that
\begin{equation}\label{Difference_set_lemma_RHS-2}|E(\mathcal{M}(A,B))|\geq
(\frac{ab}{c}-1)\frac{ab}{2}.
\end{equation}
Since $|E(\mathcal{M}(A,B))|=|E(\mathcal{M}(A,-B))|$ follows from
Proposition \ref{Graph_Bijection}, then by comparing
(\ref{Differense_set_lemma_LHS}) and
(\ref{Difference_set_lemma_RHS}), it follows that $c=|A+B|$ must
satisfy the bound
\begin{equation}\label{Difference_set_lemma_bound}|A+B|\geq
\frac{a^2b^2-\delta_0^2}{M+ab-\delta_0},\end{equation} and by
comparing (\ref{Differense_set_lemma_LHS}) and
(\ref{Difference_set_lemma_RHS-2}), it follows that
\begin{equation}\label{Difference_set_lemma_bound-2}|A+B|\geq
\frac{a^2b^2}{M+ab}.\end{equation} Noting that
$M=tb(b-k)+(k-1)ab-\delta(k-\delta)\leq tb(b-k)+(k-1)ab$, it follows
that (\ref{Difference_set_lemma_bound-2}) implies (i), and
(\ref{Difference_set_lemma_bound}) rearranges to yield (ii) (the
second two inequalities follow immediately from the first).

Suppose that
\begin{equation}\label{A-B_sym_bound}|A+B|\geq
\min\left\{\frac{|A||B|}{d},\,\frac{2|A||B|}{d+1}-\frac{M}{d(d+1)}\right\},\end{equation}
for every integer $d\geq 1$. Comparing the two bounds in the
minimum, we see that
$$\frac{2|A||B|}{d+1}-\frac{M}{d(d+1)}\leq\frac{|A||B|}{d}$$ holds for
$d\leq \frac{M+ab}{ab}$. Thus, since $M\geq 0$, then applying
(\ref{A-B_sym_bound}) with
$d=\left\lfloor\frac{M+ab}{ab}\right\rfloor\geq 1$ yields (iii). So
it remains to establish (\ref{A-B_sym_bound}).

Note that if $l=0$, then $|A+B|=|A||B|$ follows from
(\ref{c_relates_to_l}), yielding (\ref{A-B_sym_bound}). So we may
assume $l>0$. Hence comparing (\ref{Differense_set_lemma_LHS}) with
(\ref{Difference_set_lemma_RHS}) (expressed in terms of $l$ without
$\delta_0$), it follows that
\begin{equation}\label{differnce_set_lemma_l_bound} |A+B|\geq
\frac{2ab}{l+1}-\frac{M}{l(l+1)}.
\end{equation}
Considering the above bound as a function of $l$, and computing its
discrete derivative, i.e., the bound evaluated at $l$ minus the
bound evaluated at $l-1$, for $l\geq 2$, we obtain
\begin{equation}\label{diffset_der}\frac{2M+2ab-2abl}{l(l^2-1)}.\end{equation}
Noting that (\ref{diffset_der}) is non-negative for $l\leq
\frac{M+ab}{ab}$, it follows that the bound in
(\ref{differnce_set_lemma_l_bound}) monotonically increases up to
$l\leq \frac{M+ab}{ab}$.

Suppose $l>\frac{M+ab}{ab}$. Hence from (\ref{c_relates_to_l}) it
follows that $$\frac{M+ab+1}{ab}\leq l=\frac{ab+\delta_0}{c}-1\leq
\frac{ab+c-1}{c}-1<\frac{ab}{c}.$$ Thus, in view of (i), it follows
that $$\frac{a^2b^2}{M+ab}\leq c<\frac{a^2b^2}{M+ab+1},$$ a
contradiction. So $l\leq \frac{M+ab}{ab}$. Consequently, the bound
in (\ref{differnce_set_lemma_l_bound}) is monotonically increasing
for all possible values of $l$.

If $l\leq d-1$, then $|A+B|=c\geq \frac{ab}{d}$ follows from
(\ref{c_relates_to_l}). Otherwise, $l\geq d$, whence the
monotonicity of (\ref{differnce_set_lemma_l_bound}) implies
$$|A+B|\geq
\frac{2ab}{d+1}-\frac{M}{d(d+1)}.$$ Hence $|A+B|$ is greater than
the minimum of these two bounds, yielding (\ref{A-B_sym_bound}), and
completing the proof.
%
\end{proof}

If $|B|\leq k$, then the restriction $|\nu_x(A,-B)|\leq k$ for all
$x\in G\setminus T$ holds trivially. Hence the assumption $|B|>k$ is
needed to gain useful information. Likewise, the assumption $|A|\geq
|T|$ is not very restrictive, since for $|A|\leq |T|$ it is possible
that $|\nu_x(A,-B)|=|B|$ for all $x\in A-B$, in which case we obtain
only the trivial bound for $|E(\mathcal{M}(A,-B)|$ (in such cases,
an improved estimate of $|E(\mathcal{M}(A,-B)|$ would yield an
improved estimate for $|A+B|$, though we will handle this issue by
bounding $|T|$ instead). We remark that the bound given in (iii) is
minimized when $x=|A||B|$, which yields (i). However, the estimate
$\frac{|A|^2|B|^2}{M+|A||B|}$ can be improved upon, under a variety
of circumstances, by using more precise estimates for $x$. For
instance, if $0\leq |T||B|(|B|-k)-\delta(k-\delta)< |A||B|$, then
$x=|A||B|-|T||B|(|B|-k)+\delta(k-\delta)$ follows from the
definitions of $M$ and $x$.

\section{A Step Beyond KST}
\indent\indent We begin with our main result, Theorem
\ref{KST_Step_Beyond}. Note that the second part of Theorem
\ref{KST_Step_Beyond} shows that the case with $A+B$ periodic
reduces to the case when $A+B$ is aperiodic. In the case when $|G|$
is prime, then the reason for assuming the second part of the bound
in (\ref{intro_chat-1}) was to exclude the cases when $|B|$ is too
small to gain exceptional structural information. Indeed, $|A+B|\leq
|A||B|\leq |A|+|B|+r$ holds trivially for any pair $(A,B)$ with $r$
sufficiently large compared to both $|A|$ and $|B|$. We note that
each of the exceptional cases given by types (V-VII) correspond
precisely to those degenerate cases in $\Z/p\Z$ when
$\min\{|A|,\,|B|\}$ is very small, lifted via a quasi-periodic
decomposition in precisely the same way the elementary pairs of
types (I-IV) were lifted for KST. The last additional exception
given by type (VIII) is the first case where a non-quasi-periodic
example does not directly correspond to behavior in $\Z/p\Z$.
However, the structure of a type (VIII) pair is highly restricted,
including
$\db(A+B,\mathcal{P}),\,\db(\overline{A},\mathcal{P}),\,\db(\overline{B},\mathcal{P})\leq
2$, and cannot occur in a cyclic group. In essence, a type (VIII)
pair $(A_0,B_0)$ has both $A_0$ and $B_0$ being arithmetic
progressions (with common difference) of cosets of a Klein four
subgroup $H_b$, with the exception that all four end terms contain
only two (rather than four) elements from the coset, with these
elements in the first term of both progressions corresponding to two
cosets of the same order two subgroup $H_{d_1}$ of $H_b$, and these
elements in the last term of both progressions corresponding to two
cosets of a different order two subgroup $H_{d_2}$ of $H_b$.

The degenerate cases in $G=\Z/p\Z$ that correspond to when $A+B$ is
very close to $G$ (for instance, type (IV) pairs), also have
corresponding analogs whenever $\db(A+B,\mathcal{P})$ is very small.
All this leads one to wonder if $|A+B|=|A|+|B|+r$ and
$\db(\overline{A},\mathcal{P}),\,\db(\overline{B},\mathcal{P}),\,\db(A+B,\mathcal{P})\geq
r+3$, with equality holding for at most one of the three quantities
(note that $|A|\geq \db(\overline{A},\mathcal{P})$, with equality
holding when $|G|$ is prime), would always imply there exists a pair
$(A',B')$ with $|A'+B'|=|A'|+|B'|-1$, such that
$\db(A,A'),\,\db(B,B')\leq r+1$ (or better yet, a pair $(A',B')$
with $|A'+B'|=|A'|+|B'|+r-1$, such that $\db(A,A'),\,\db(B,B')\leq
1$). If true, this would mean that either one of $\overline{A}$,
$\overline{B}$ or $A+B$ is close to being periodic, or else
$|A+B|-|A|-|B|$ bounds how far the pair $(A,B)$ can be from being a
critical pair $(A',B')$.

\begin{thm}\label{KST_Step_Beyond} Let $A$ and $B$ be finite, nonempty
subsets of an abelian group $G$ with $|A+B|=|A|+|B|$. If $A+B$ is
aperiodic, then either there exist $\alpha,\,\beta\in G$ such that
\begin{equation}\label{Kemp_extendible}|A\cup\{\alpha\}+B\cup\{\beta\}|=
|A\cup\{\alpha\}|+|B\cup\{\beta\}|-1,\end{equation} or there exist
quasi-period decompositions $A=A_1\cup A_0$ and $B=B_1\cup B_0$ with
common quasi-period $H_a$, and $A_0$ and $B_0$ nonempty, such that

$(i)$ $|\nu_c(\phi_a(A),\phi_a(B))|=1$, where  $c=
\phi_a(A_0)+\phi_a(B_0)$

$(ii)$ $|\phi_a(A)+\phi_a(B)|=|\phi_a(A)|+|\phi_a(B)|-1$,

$(iii)$ the pair $(A_0,B_0)$ is of one of the following types (each of which imply $|A_0+B_0|=|A_0|+|B_0|$):\\
(V) $|A_0|=2$ or
$|B_0|=2$, and $|A_0+B_0|=|A_0|+|B_0|\leq |H_a|-2$;\\
(VI) $|A_0|=|B_0|=3$, $A_0=b_0+B_0$ for some $b_0\in G$, and
$|2A_0|>2|A_0|-1$;\\
(VII) $|A_0|=3$ or $|B_0|=3$, and either $|2A_0|> 2|A_0|-1$ and
$B_0=b_0+((2a_0-\overline{2A_0})\cap H_a)$ for some $a_0\in A_0$ and
$b_0\in B_0$ (if $|A_0|=3$), or $|2B_0|>2|B_0|-1$ and
$A_0=a_0+((2b_0-\overline{2B_0})\cap H_a)$ for some $a_0\in A_0$ and
$b_0\in B_0$ (if $|B_0|=3$),
and in both cases $|A_0+B_0|=|H_a|-3\geq 6$;\\
(VIII) there exists a subgroup $H_b\leq H_a$, with $H_b\cong
\Z/2\Z\times \Z/2\Z$, such that both $\phi_b(A_0)$ and $\phi_b(B_0)$
are arithmetic progressions with common difference $\phi_b(d)$,
where the order of $\phi_b(d)$ is at least
$|\phi_b(A_0)|+|\phi_b(B_0)|-1$ and
$|\phi_b(A_0)|,\,|\phi_b(B_0)|\geq 2$, such that $a_i+H_b\subseteq
A_0$ and $b_i+H_b\subseteq B_0$ for all non-end terms
$\phi_b(a_i)\in \phi_b(A_0)$ and $\phi_b(b_i)\in \phi_b(B_0)$, such
that $(a_0+H_b)\cap A_0$ and $(b_0+H_b)\cap B_0$ are each an
$H_{d_1}$-coset, and such that $(a_l+H_b)\cap A_0$ and
$(b_{l'}+H_b)\cap B_0$ are each an $H_{d_2}$-coset, where $H_{d_1}$
and $H_{d_2}$ are distinct cardinality two subgroups of $H_b$, where
$\phi_b(a_0)$ and $\phi_b(b_0)$ are each the first term in
$\phi_b(A_0)$ or $\phi_b(B_0)$, respectively, and where
$\phi_b(a_l)$ and $\phi_b(b_{l'})$ are each the last term in
$\phi_b(A_0)$ or $\phi_b(B_0)$, respectively (whence
$\db(C,\mathcal{QP}_{H_{d_i}})=1$ and
$\db(C,\mathcal{P}_{H_{d_i}})=2$, for all $C\in
\{A,B,A+B,\overline{A},\overline{B},\overline{A+B}\}$ and $i=1,2$).

If $A+B$ is periodic with maximal period $H_a$, then either
(\ref{Kemp_extendible}) holds, or else $A$ and $B$ are
$H_a$-periodic, $\phi_a(A+B)$ is aperiodic, and
$|\phi_a(A)+\phi_a(B)|=|\phi_a(A)|+|\phi_a(B)|$.
\end{thm}

Observe that (\ref{Kemp_extendible}) implies that a pair $(A,B)$
with $|A+B|=|A|+|B|$ is a large subset of a critical pair $(A',B')$
(with at most one hole in each set). One might also like to know
where such holes $\alpha$ and $\beta$ can be placed in a critical
pair $(A',B')$.

If the pair $(A,B)$ is extendible, then this question is easily
answered using KST as follows. The extendibility of $(A,B)$ implies
that (\ref{Kemp_extendible}) holds with $A'+B'=A+B$, i.e, that only
one element $\alpha$ was deleted, say from $A'$ (i.e., $\beta\in
B$), whence there cannot have been any unique expression element of
the form $\alpha+b$ in $A'+B'$. If $|\nu_x(A',B')|\geq 2$ for all
$x$, then $\alpha$ can be any element. Thus let $A'=A'_1\cup A'_0$
and $B'=B'_1\cup B'_0$ be the Kemperman decompositions with common
quasi-period $H_a$. In view of KST(iii), $\alpha\in A$ can be any
element from $A'_1$. If we have type (I), then $\alpha\notin A'_0$;
if type (II), then $\alpha$ can be any element in $A'_0$ except the
two end terms of the arithmetic progression; if type (III), then
$A'_0=(g_0-(\overline{B'_0}\cap (g_1+H_a)))\cup\{g_0-g_1\}$, where
$g_1\in B'_0$, and $\alpha$ can be any element except $g_0-g_1$; and
if type (IV), then $\alpha$ can be any element of $A'_0$.

If the pair $(A,B)$ is non-extendible, then the answer is more
complicated, and is provided by the following corollary.

\begin{cor}\label{nec-suff-cond} Let $A$ and $B$ be finite, nonempty
subsets of an abelian group $G$ with $|A+B|=|A|+|B|$, and suppose
that $(16)$ holds. Let $A'=A\cup \{\alpha\}$ and let $B'=B\cup
\{\beta\}$. If $(A,B)$ is non-extendible, then $A'+B'$ cannot be
periodic without a unique expression element. Hence we can let
$A'=A'_1\cup A'_0$ and $B'=B'_1\cup B'_0$ be the Kemperman
decompositions with common quasi-period $H_a$. Furthermore, one of
the following must also hold:

(A) $(A',B')$ has type (II) with both $\alpha\in A'_0$ and $\beta\in
B'_0$, both $\alpha$ and $\beta$ are the second term (by appropriate
choice of sign for the progression) in their arithmetic progression
with difference $d$, $|A'_0|,\,|B'_0|\geq 3$ with equality holding
for at most one of the two, and $|\langle d\rangle|\geq
|A'_0|+|B'_0|-1$;

(B) $(A,B)$ is obtainable by applying Proposition \ref{AB_dual} to a
non-extendible type (V) pair $(A_0,B_0)$ from $H_a\leq G$, and
inserting this pair (via appropriate translation) into the aperiodic
parts of a quasi-periodic decomposition with quasi-period $H_a$ that
satisfies KST modulo $H_a$;

(C) $(A',B')$ has type (I) such that both the aperiodic parts in the
Kemperman decompositions have cardinality one, $\alpha\in A'_1$,
$\beta\in B'_1$, and $(\phi_a(A'),\phi_a(B'))$ has type (II) with
both $\phi_a(A'_{\alpha})$ and $\phi_a(B'_{\beta})$ contained in the
aperiodic part of the mod $H_a$ Kemperman decomposition, and if
$|H_a|>2$, then both are the second term in their arithmetic
progression with $A'_0+\beta=B'_0+\alpha=\gamma\notin A+B$, while if
$|H_a|=2$, then either both are the first term in their progression,
and $A'_0+\beta\neq B'_0+\alpha$ if both progressions have length
two, or else both are the second term in their arithmetic
progression, $A'_0+\beta=B'_0+\alpha=\gamma\notin A+B$, and one of
the progressions has at least three terms (in all cases, by
appropriate choice of sign for the progressions)
\end{cor}


\begin{proof}
Since $(A,B)$ is non-extendible and (\ref{Kemp_extendible}) holds,
it follows that ($A'+B')\setminus \gamma=A+B$, with $\gamma\in
A'+B'$. Furthermore, we may assume that
$|N_1^\beta(A',B')|=|N_1^\alpha(B',A')|=0$, since otherwise the
non-extendibility of $(A,B)$ is contradicted by either the critical
pair $(A',B)$ or the critical pair $(A,B')$. Thus $\gamma$ must have
exactly two distinct representations $\alpha+\beta'$ and
$\alpha'+\beta$ in $A'+B'$.

Suppose $\alpha\in A'_0$ and $\beta\in B'_0$ (we assume for the
moment that $A'_1\cup A'_0$ and $B'_1\cup B'_0$ exist, which will be
shown to be the case in a following paragraph by similar arguments),
and let $A_0=A'_0\setminus \alpha$ and $B_0=B'_0\setminus \beta$.
Hence, since $|N_1^\beta(A',B')|=|N_1^\alpha(B',A')|=0$, it follows
that $(A',B')$ cannot have type (I). If we have type (II), then
$A'_0$ and $B'_0$ are arithmetic progressions with common difference
$d$, and by choosing an appropriate sign for $d$, it follows that
both $\alpha$ and $\beta$ must be the second term in their
respective progression. Furthermore,
$|N_1^\beta(A',B')|=|N_1^\alpha(B',A')|=0$ implies that
$|A'_0|,\,|B'_0|\geq 3$, and KST implies that $|\langle
d\rangle|\geq |A'_0|+|B'_0|-1$. Additionally, $|A'_0|=|B'_0|=3$
cannot hold, else $|(A'+B')\setminus (A+B)|\geq 2$, a contradiction.
Thus (A) holds.

If we have type (IV), then $|A_0+B_0|=|H_a|-2$, whence in view of
Proposition \ref{AB_dual} it follows that $(A_0,B_0)$ is just the
result of applying Proposition \ref{AB_dual} with $G=H_a$ to a
non-extendible type (V) pair from the group $H_a$, and translating
appropriately (note type (VI) and (VII) pairs have a similar dual
relationship, as do type (I) and (IV) pairs). Thus (B) holds. If we
have type (III), then $\db(A+B,\mathcal{P})=1$ and
$|A_0|+|B_0|=|H_a|-1$, with both $A_0$ and $B_0$ contained in an
$H_a$-coset. Note that there are exactly $|H_a|-|B_0|$ elements
$x\in A_0+H_a$ such that $\gamma\notin x+B_0$. Since
$|A_0|<|H_a|-|B_0|$, it follows that there must be at least one such
$x\in \overline{A_0}\cap (A_0+H_a)$, whence adding $x$ contradicts
the non-extendibility of $(A,B)$. Thus type (III) cannot occur for
$(A',B')$ in this situation.

Next suppose $A'+B'$ is periodic without a unique expression
element. Thus $A+B\cup\{\gamma\}$ is periodic with maximal period
$H_a$, and there are exactly $|H_a|+1$ holes in $A$ and $B$. Since
$\gamma\notin A+B$, it follows in view of Proposition
\ref{mult_result} applied with $G=H_a$ that there must exist $a_0\in
A$ and $b_0\in B$ such that $A_{a_0}$ and $B_{b_0}$ contain at least
$|H_a|$ holes with $\gamma\in A_{a_0}+B_{b_0}+H_a$. Thus if the last
remaining hole is not also in either $A_{a_0}$ or $B_{b_0}$, then
adding it to either $A$ or $B$ will contradict the non-extendibility
of $(A,B)$. Hence $|A_{a_0}|+|B_{b_0}|=|H_a|-1$, whence the
arguments used in the previous paragraph for type (III) show that
the pair $(A,B)$ is extendible. Thus $A'+B'$ cannot be periodic
without a unique expression element. Thus, in view of previous
cases, we may suppose $\alpha\in A'_1$ or $\beta\in B'_1$, and
w.l.o.g. assume the former.

Suppose $|H_a|=2$. Thus $\db(A+B,\mathcal{P})\leq 2$ and either
$|A'_0|=1$ or $|B'_0|=1$. Hence $(A',B')$ has type (I), whence
$|N_1^\beta(A,B)|=0$ implies $\beta\in B'_1$ as well. Since $\gamma$
has exactly two representations in $A'+B'$, since every element of
$A'_0+B'_0$ is a unique expression element, and since
$\phi_a(A'_0)+\phi_a(B'_0)$ is a unique expression element, it
follows that $\phi_a(A'_0)+\phi_a(B'_0)\neq \phi_a(\gamma)$. Hence,
since $\gamma\notin A+B$, it follows that
$\phi_a(A'_0)+\phi_a(B'_\beta)$, $\phi_a(B'_0)+\phi_a(A'_\alpha)$,
and $\phi_a(A'_{\alpha})+\phi_a(B'_{\beta})$ are the only possible
representations of $\phi_a(\gamma)\in \phi_a(A')+\phi_a(B')$.

Suppose $\phi_a(A'_{\alpha})+\phi_a(B'_{\beta})$ is a unique
expression element. Hence, since $\phi_a(A'_0)+\phi_a(B'_0)$ is also
a unique expression element, and since $\alpha\notin A'_0$ and
$\beta\notin B'_0$, it follows that $(\phi_a(A'),\phi_a(B'))$ must
have type (II) with (by choosing the appropriate sign for the
progression) $\phi_a(A'_0)$ the first term in the arithmetic
progression and $\phi_a(A'_{\alpha})$ the last term in the
arithmetic progression, and the same holding true for $\phi_a(B'_0)$
and $\phi_a(B'_{\beta})$. Since
$\phi_a(A'_{\alpha})+\phi_a(B'_{\beta})$ is a unique expression
element, and since $|H_a|=2$, it follows that $(A'_\alpha\setminus
\alpha)+(B'_\beta\setminus\beta)$ must be missing an element from
the coset $\alpha+\beta+H_a$, which must then be $\gamma$ (since
deleting $\alpha$ and $\beta$ removes only the single element
$\gamma$ from $A'+B'$). Hence
$\phi_a(A'_\alpha)+\phi_a(B'_\beta)=\phi_a(\gamma)$. Thus, since
$\phi_a(A'_\alpha)+\phi_a(B'_\beta)=\phi_a(\gamma)$ is a unique
expression element, it follows that $|A'_0|=|B'_0|=1$, else adding
the other element from the $H_a$-coset either to $A'_0$ or $B'_0$
will contradict the non-extendibility of $(A,B)$. Furthermore, if
there are only two terms in each progression, then it follows that
$(A'_0+(B'_\beta\setminus\beta))\cup (B'_0+(A'_\alpha\setminus
\alpha))$ is an $H_a$-coset, else $A+B$ will be missing an element
besides $\gamma$ from the coset
$A'_0+B'_\beta+H_a=B'_0+A'_\alpha+H_a$, contradicting that
$(A'\setminus \alpha)+(B'\setminus \beta)=(A'+B')\setminus \gamma$.
Hence $A'_0+\beta\neq B'_0+\alpha$ in this case. Thus (C) holds. So
next assume $\phi_a(A'_{\alpha})+\phi_a(B'_{\beta})$ is not a unique
expression element.

If $\phi_a(A'_\alpha)+\phi_a(B'_\beta)=\phi_a(\gamma)$, then there
will be two representations of $\gamma$ contained in
$A'_\alpha+B'_\beta$ (since $\alpha\in A'_1$ and $\beta\in B'_1$),
whence $\phi_a(A'_\alpha)+\phi_a(B'_\beta)$ not a unique expression
element implies that $\gamma$ must have at least one more
representation in $A'+B'$ (since $\phi_a(A'_0)+\phi_a(B'_0)$ is a
unique expression element, and hence not equal to
$\phi_a(A'_\alpha)+\phi_a(B'_\beta)=\phi_a(\gamma)$, and since all
other pairs $\phi_a(b_1)\in \phi_a(A')$ and $\phi_a(b_2)\in
\phi_a(B')$ have either $b_1+H_a\subseteq A'$ or $b_2+H_a\subseteq
B'$), a contradiction. Therefore
$\phi_a(A'_\alpha)+\phi_a(B'_\beta)\neq \phi_a(\gamma)$.

Thus, since $\gamma$ has exactly two representations in $A'+B'$
given by $\alpha+\beta'$ and $\alpha'+\beta$, and since
$\phi_a(A'_0)+\phi_a(B'_\beta)$, $\phi_a(B'_0)+\phi_a(A'_\alpha)$,
and $\phi_a(A'_{\alpha})+\phi_a(B'_{\beta})$ are the only possible
representations of $\phi_a(\gamma)\in \phi_a(A')+\phi_a(B')$, it
follows  that
$\phi_a(A'_0)+\phi_a(B'_{\beta})=\phi_a(B'_0)+\phi_a(A'_{\alpha})=\phi_a(\gamma)$
is an element of $\phi_a(A')+\phi_a(B')$ with exactly two
representations, and that $|A'_0|=|B'_0|=1$. From KST it follows, as
remarked in Section 2, that if $(A',B')$ has type (I) with
$|A'_0|=1$, then $(\phi_a(A'),\phi_a(B'))$ cannot have type (I) with
the aperiodic part of $\phi_a(A')$ having cardinality one.

Let $H_d/H_a$ be the quasi-period from KST applied modulo $H_a$ to
$(\phi_a(A'),\phi_a(B'))$. Hence, since
$|\phi_a(A')_{\phi_a(b),H_d/H_a}|\geq 2$ for all $b\in A'$ (in view
of the previous paragraph), it follows that
$$|(\phi_a(\beta)+H_d/H_a)\cap (\phi_a(\gamma)-\phi_a(A'))|\geq 2.$$
Thus, since
$\phi_a(A'_0)+\phi_a(B'_{\beta})=\phi_a(B'_0)+\phi_a(A'_{\alpha})=\phi_a(\gamma)$
is an element of $\phi_a(A')+\phi_a(B')$ with exactly two
representations, it follows that either $\phi_a(B'_{\beta})$ and
$\phi_a(B'_0)$ are contained in the same $H_d/H_a$-coset, or else
$\phi_a(B'_\beta)$ must be contained in the $H_d/H_a$-coset
corresponding to the aperiodic part of the Kemperman decomposition.
However, since $\phi_a(B'_0)$ must be contained in the
$H_d/H_a$-coset corresponding to the aperiodic part of the Kemperman
decomposition (a consequence of KST(i)(iii)), it follows that
$\phi_a(B'_{\beta})$ is also contained in the aperiodic part the
Kemperman decomposition modulo $H_a$ in the former case as well.
Likewise for $\phi_a(A'_{\alpha})$.

If $(\phi_a(A'),\phi_a(B'))$ has type (II), then (by appropriate
choice of sign), $\phi_a(A'_0)$ is the first term in the arithmetic
progression, and $\phi_a(A'_{\alpha})$ is the second term, with the
same true of $\phi_a(B'_0)$ and $\phi_a(B'_{\beta})$ (in view of the
previous paragraph, and since
$\phi_a(A'_\alpha)+\phi_a(B'_0)=\phi_a(A'_0)+\phi_a(B'_\beta)$ is an
element with exactly two distinct representations in
$\phi_a(A')+\phi_a(B')$). Furthermore,
$\phi_a(A'_\alpha)+\phi_a(B'_\beta)$ not a unique expression element
implies that $\phi_a(A'_\alpha)$ and $\phi_a(B'_\beta)$ cannot also
both be end terms, whence there must be at least three terms in one
of the progressions. Additionally, since
$\phi_a(B'_0)+\phi_a(A'_\alpha)=\phi_a(A'_0)+\phi_a(B'_\beta)=\phi_a(\gamma)$,
it follows that $\gamma\notin (A'_0+(B'_\beta\setminus\beta))\cup
(B'_0+(A'_\alpha\setminus\alpha))$, whence
$A'_0+\beta=B'_0+\alpha=\gamma$. Thus (C) holds.

Note $(\phi_a(A'),\phi_a(B'))$ cannot have type (IV), nor, as
remarked earlier, type (I). If instead $(\phi_a(A'),\phi_a(B'))$ has
type (III), then
$\phi_a(A'_0)+\phi_a(B'_{\beta})=\phi_a(B'_0)+\phi_a(A'_{\alpha})$
being an element of $\phi_a(A')+\phi_a(B')$ with exactly two
representations would imply $\phi_a(A'_0)+\phi_a(B'_{\beta})$ was a
unique expression element in $\phi_a(A')+\phi_a(B'\setminus B'_0)$.
Since a type (III) pair has exactly one unique expression element,
it follows that $|\phi_a(A'_0)|,\,|\phi_a(B'_0)|\geq 3$. Hence,
since $\phi_a(A'_0)+\phi_a(B'_0)$ is the unique expression element
in $\phi_a(A')+\phi_a(B')$, it follows from KST that
$(\phi_a(A'),\phi_a(B'\setminus B'_0))$ either has type (IV), whence
there cannot be any unique expression element, a contradiction, or
else type (II). In the latter case, then
$\phi_a(A'_0)+\phi_a(B'_{\beta})$ a unique expression element in
$(\phi_a(A')+\phi_a(B'\setminus B'_0))$ implies that $\phi_a(A'_0)$
is an end term of the arithmetic progression, whence
$\phi_a(A'_0)+\phi_a(B'_0)$ a unique expression element in
$\phi_a(A')+\phi_a(B')$ implies that $\phi_a(B'_0)$ must follow
directly after/before the end term of the progression
$\phi_a(B'\setminus B'_0)$, whence $(\phi_a(A'),\phi_a(B'))$ has
type (II), a contradiction. Thus type (III) cannot occur for
$(\phi_a(A'),\phi_a(B'))$ in this situation.

Finally, assume instead $|H_a|>2$. Since $\alpha\in A'_1$, then from
Proposition \ref{mult_result}, it follows that
$(A'_{\alpha}\setminus \alpha)+B_{b_i}$ will be $H_a$-periodic for
all $|B_{b_i}|\geq 2$. Thus $\gamma\notin A+B$ and $|H_a|>2$ imply
that $\phi_a(A'_\alpha)+\phi_a(B'_0)=\phi_a(\gamma)$, and that
either $|B'_0|=1$ or else $|B'_0|=2$ and $\beta\in B'_0$. However,
in the latter case, $|B'_0|=2$ implies that type (I) or (II) must
hold (since type (III) and (IV) both imply that each of the
aperiodic parts of the Kemperman decompositions have cardinality at
least three), whence $|B'_0|=2$ further implies that
$|N_1^b(A,B)|>0$ for all $b\in B'_0$, contradicting that
$|N_1^\beta(A,B)|=0$. Therefore we can assume $|B_0|=1$, whence we
must have type (I). Hence $|N_1^\beta(A,B)|=0$ implies that $\beta
\in B_1$ also. Applying the same arguments with the roles of $A$ and
$B$ reversed, it follows that $|B_0|=|A_0|=1$ with
$\phi_a(A'_0)+\phi_a(B'_{\beta})=\phi_a(B'_0)+\phi_a(A'_{\alpha})=\phi_a(\gamma)$
an element of $\phi_a(A'+B')$ with exactly two representations
(since $|\nu_\gamma(A',B')|=2)$, whence the arguments from the
previous three paragraphs apply for determining the structure modulo
$H_a$, with the exception that the arithmetic progressions from KST
are allowed to both have length two, yielding (C), and completing
the proof.\end{proof}

We remark that the sufficiency of the examples (A-C) is easily
checked, as is the sufficiency for types (V-VIII) (in view of
Proposition \ref{mult_result}, Lemma \ref{kemp_Lemma_qp}, and KST).
Thus together, along with the description provided for extendible
pairs $(A,B)$, they may be taken as necessary and sufficient
conditions for $|A+B|=|A|+|B|$. Simple consequences of the above
description and Theorem \ref{KST_Step_Beyond} are the following two
immediate corollaries.

\begin{cor}\label{cor1} Let $A$ and $B$ be finite, nonempty, and non-extendible
subsets of an abelian group $G$ such that $|A+B|=|A|+|B|$. If $A$
and $B$ are non-quasi-periodic, generating subsets such that
$|A|,\,|B|,|\overline{A+B}|\geq 3$, with equality holding for at
most one of the three, then
$$\db(A+B,\mathcal{S})=\db(\overline{A+B},\mathcal{S})=1,\mbox{ and}$$
$$\db(B,\mathcal{S})=
\db(\overline{B},\mathcal{S})=\db(A,\mathcal{S})=\db(\overline{A},\mathcal{S})=1,$$
where either $\mathcal{S}=\mathcal{QP}_H$, for a nontrivial, proper
subgroup $H$, or else $\mathcal{S}=\mathcal{AP}_d$, for a nonzero
$d\in G$.\end{cor}

\begin{cor}\label{cor2} Let $A$ and $B$ be finite, nonempty
subsets of an abelian group $G$ such that $|A+B|=|A|+|B|$ and
$\db(A,\mathcal{QP}),\,\db(B,\mathcal{QP})\geq 2$. If
$\db(\overline{A},\mathcal{P}),\,\db(\overline{B},\mathcal{P}),\,
\db(A+B,\mathcal{P})\geq 3$, with equality holding for at most one
of the three, then for some nonzero $d\in G$,
$$\db(A+B,\mathcal{AP}_d),\,
\db(A,\mathcal{AP}_d),\,\db(B,\mathcal{AP}_d)\leq 1.$$
\end{cor}

The proof of Theorem \ref{KST_Step_Beyond} is heavily reliant upon a
series of reductions to simpler cases. An important step in the
proof will be to show that it suffices to prove Theorem
\ref{KST_Step_Beyond} when both $A$ and $B$ are non-quasi-periodic,
generating subsets. This will be accomplished principally through
the following three lemmas.


\begin{lem}\label{kemp_firstlemma} Let $A$ and $B$ be finite, nonempty
subsets of an abelian group $G$ with $|A|\geq 3$, $|A+B|=|A|+|B|$,
and $0\in A$, and let $H_a=\langle A\rangle$. If $A+B$ is aperiodic
and $B$ is non-extendible, then $B$ has a quasi-periodic
decomposition with quasi-period $H_a$.
\end{lem}

\begin{proof}
Let $B'$ be the maximal subset of $B$ that is $H_a$-periodic, and
let $B\setminus B'=B_{b_1}\cup\ldots\cup B_{b_l}$ be an $H_a$-coset
decomposition of $B\setminus B'$. From the maximality of $B'$ it
follows that no $B_{b_i}$ is $H_a$-periodic. Hence, since $B$ is
non-extendible, it follows that
\begin{equation}\label{kemp_lemma1_3}A+B_{b_i}\neq b_i+H_a,\end{equation}
for all $i$. Since $|A+B|=|A|+|B|$, since $B'$ is $H_a$-periodic,
and since $|\phi_a(A)|=1$, it follows that
\begin{equation}\label{kemp_lemma1_2}|A+(B\setminus B')|=|A|+|B\setminus B'|.\end{equation}
If the lemma is false, then $l\geq 2$. Thus $$|A+(B\setminus
B')|=|\cup_{i=1}^{l} A+B_{b_i}|\geq
|A+B_{b_1}|+|A+B_{b_2}|+|B\setminus \{B'\cup B_{b_1}\cup
B_{b_2}\}|,$$ implying in view of (\ref{kemp_lemma1_2}) that
\begin{equation}\label{kemp_lemma1_1}|A+B_{b_1}|+|A+B_{b_2}|=
|A+(B_{b_1}\cup B_{b_2})|\leq |A|+|B_{b_1}|+|B_{b_2}|.\end{equation}
In view of Kneser's Theorem, it follows that \be\label{lem1-1}
|A+B_{b_i}|\geq |A|+|B_{b_i}|-|H_{a_i}|+\rho_i,\ee where $A+B_{b_i}$
is maximally $H_{a_i}$-periodic and $\rho_i$ is the number of
$H_{a_i}$-holes in $A$ and $B_{b_i}$. Since $H_a=\langle A\rangle$,
it follows in view of (\ref{kemp_lemma1_3}) that
$|\phi_{a_i}(A)|\geq 2$, whence \be\label{lem1-2}|A|\geq
|\phi_{a_i}(A)||H_{a_i}|-\rho_i\geq 2|H_{a_i}|-\rho_i.\ee Combining
(\ref{lem1-1}) and (\ref{kemp_lemma1_1}), and w.l.o.g. assuming
$|H_{a_1}|\geq |H_{a_2}|$, it follows that \be\label{onemore}|A|\leq
|H_{a_1}|+|H_{a_2}|-\rho_1-\rho_2\leq 2|H_{a_1}|-\rho_1-\rho_2.\ee
Thus in view of (\ref{lem1-2}) applied with $i=1$, it follows that
$\rho_2=0$. Hence, since $A$ is aperiodic (else $A+B$ is periodic),
it follows that $|H_{a_2}|=1$, whence (\ref{onemore}) and
(\ref{lem1-2}) imply $|H_{a_1}|=1$ also. Thus (\ref{onemore})
implies $|A|\leq 2$, a final contradiction.\end{proof}


\begin{lem}\label{kemp_Lemma_qp}Let $A$ and $B$ be finite, nonempty
subsets of an abelian group $G$ with $A+B$ aperiodic and $(A,B)$
non-extendible, and let $A=A_1\cup A_0$ be a quasi-periodic
decomposition with $A_1$ nonempty and periodic with maximal period
$H_a$. If $|A+B|=|A|+|B|$, then $B$ has a quasi-periodic
decomposition $B=B_1\cup B_0$ with quasi-period $H_a$, such that:

(i) $|\nu_c(\phi_a(A),\phi_a(B))|=1$, where $c=\phi_a(A_0+B_0)$

(ii) $|\phi_a(A+B)|=|\phi_a(A)|+|\phi_a(B)|-1$,

(iii) $|A_0+B_0|=|A_0|+|B_0|$.
\end{lem}

\begin{proof} 

Let $B'$ be the maximal subset of $B$ that is $H_a$-periodic, and
let $B\setminus B'=B_{b_1}\cup\ldots\cup B_{b_l}$ be an $H_a$-coset
decomposition of $B\setminus B'$. From the maximality of $B'$ it
follows that no $B_{b_i}$ is $H_a$-periodic. Hence, since $B$ is
non-extendible, it follows, for all $i$, that
\begin{equation}\label{kemp_lemma2_1}A_0+B_{b_i}\neq b_i+H_a.\end{equation}

If $|\phi_a(A_1)+\phi_a(B)|<|\phi_a(A_1)|+|\phi_a(B)|-1$, then from
Kneser's Theorem it follows that $\phi_a(A_1)+\phi_a(B)$ is
periodic, contradicting either the maximality of $H_a$ for $A_1$ or
the fact that $A$ is non-extendible. Therefore
\begin{equation}\label{transfer_lemma_1}|\phi_a(A_1)+\phi_a(B)|\geq |\phi_a(A_1)|+
|\phi_a(B)|-1.\end{equation} Since $A+B$ is aperiodic, it follows
that $|H_a||\phi_a(B)|>|B|$ and that $|A+B|\geq |A_1+B|+|A_0|$.
Hence if (\ref{transfer_lemma_1}) is strict, then
\begin{equation}\label{kemp_lem2_3}|A+B|\geq
|H_a|(|\phi_a(A_1)|+|\phi_a(B)|)+|A_0|\geq |A|+|B|+1,\end{equation}
a contradiction. Therefore we can assume
\begin{equation}\label{transfer_lemma_2}|\phi_a(A_1)+\phi_a(B)|=
|\phi_a(A_1)|+|\phi_a(B)|-1.\end{equation} Likewise, if
$A_0+B'\nsubseteq A_1+B$, then $|A+B|\geq |A_1+B|+|H_a|+|A_0|$,
whence (\ref{kemp_lem2_3}) again follows, a contradiction. So we may
assume otherwise.

From the non-extendibility of $B$, and in view of
(\ref{kemp_lemma2_1}), it follows that $A_0+B_{b_i}$ is disjoint
from $A_1+B$ for all $i$. Thus $\phi_a(A_0)+\phi_a(B_{b_i})$ is a
unique expression element for each $i$. Hence, since
$A_0+B'\subseteq A_1+B$, it follows in view of
(\ref{transfer_lemma_2}) that
\begin{equation}\label{transfer_6}|\phi_a(A+B)|=|\phi_a(A)|+|\phi_a(B)|-2+l.\end{equation}
Note (i) and (ii) along with $|A+B|=|A|+|B|$ force (iii) by a simple
counting argument. Consequently, the proof will be complete if
$l=1$. So assume $l\geq 2$.

From Kneser's Theorem applied via translation with group $H_a$, it
follows that
\begin{equation}\label{transer_3}|A_0+B_{b_i}|\geq
|A_0|+|B_{b_i}|-|H_{a_i}|,\end{equation} where $A_0+B_{b_i}$ is
maximally $H_{a_i}$-periodic, and $H_{a_i}\leq H_a$. In view of
(\ref{transfer_lemma_2}), and since each $A_0+B_{b_i}$ is disjoint
from $A_1+B$, it follows that
\begin{multline}\label{lem2-2}
|A+B|\geq
(|\phi_a(A_1)|+|\phi_a(B')|+l-1)|H_a|+\Sum{i=1}{l}|A_0+B_{b_i}|=
\\ |A|+|B'|-|A_0|+(l-1)|H_a|+ \Sum{i=1}{l}|A_0+B_{b_i}|\geq \\
|A|+|B|-|A_0|-|B_{b_1}|-|B_{b_1}|+(l-1)|H_a|+
|A_0+B_{b_1}|+|A_0+B_{b_2}|.
\end{multline}
In view of (\ref{kemp_lemma2_1}), it follows that $|H_{a_i}|\leq
\frac{1}{2}|H_a|$. Thus (\ref{transer_3}), (\ref{lem2-2}) and $l\geq
2$ together imply that $$|A+B|\geq |A|+|B|+|A_0|,$$ whence $|A_0|=0$
and $A=A_1$, contradicting that $A+B$, and thus $A$, is aperiodic.
\end{proof}


\begin{lem}\label{kemp_generate_lem}Let $A$ and $B$ be finite, nonempty
subsets of an abelian group $G$ with $|A+B|=|A|+|B|$, $0\in A\cap
B$, and $|A|,\,|B|\geq 3$. If $A+B$ is aperiodic, $(A,B)$ is
non-extendible, and neither $A$ nor $B$ is quasi-periodic, then
$\langle A\rangle =\langle B\rangle$.
\end{lem}

\begin{proof}
Since $B$ is not quasi-periodic, it follows from Lemma
\ref{kemp_firstlemma} that $\langle B\rangle \leq \langle A\rangle$.
Likewise $\langle A\rangle \leq \langle B\rangle$, whence $\langle
A\rangle =\langle B\rangle$.
\end{proof}


The following lemma essentially shows that it is sufficient for $A$
to be a non-quasi-periodic, generating subset in order that this be
true (at least when $G$ is finite) of every
$C\in\{A,B,A+B,\overline{B},\overline{A},\overline{A+B}\}$. In the
proof, via various means, we will often reduce the case $A+B=C$ to a
case $A'+B'=C'$, where at least one of $A'$ $B'$ and $C'$ is a set
from $\pm\{A,B,A+B,\overline{B},\overline{A},\overline{A+B}\}$, and
then employ an induction hypothesis to $A'+B'=C'$, or the like.
However, since we will want to stay restricted to the class of
non-quasi-periodic, generating subsets, the following lemma, along
with Proposition \ref{AB_dual}, will allow us to transfer these
assumptions from $A+B=C$ to $A'+B'=C'$.

\begin{lem} \label{kemp_lem_gen2}Let $A$ and $B$ be finite, nonempty
subsets of an abelian group $G$ with $|A+B|=|A|+|B|$, $0\in A\cap
B$, $|A|,\,|B|,\,|\overline{A+B}|\geq 3$, $A+B$ aperiodic and
$(A,B)$ non-extendible. If $\langle A\rangle =G$ and $A$ is not
quasi-periodic, then $B$ is not quasi-periodic and $\langle
B\rangle=G$. Furthermore, if $G$ is finite, then neither $A+B$ nor
$\overline{A+B}$ is quasi-periodic, and $\langle
-\gamma+\overline{A+B}\rangle=G$, where $\gamma\in\overline{A+B}$.
\end{lem}

\begin{proof}
If $B$ is quasi-periodic with quasi-period $H_a$, then since $A$ is
not quasi-periodic, it follows in view of Lemma \ref{kemp_Lemma_qp}
that $G=\langle A\rangle \leq H_a$, implying $H_a=G$.  Hence $B=G$,
contradicting that $A+B$ is aperiodic. Therefore we can assume $B$
is not quasi-periodic. Hence from Lemma \ref{kemp_generate_lem} it
follows that $G=\langle A\rangle=\langle B\rangle$.

Now assume $G$ is finite. Since $(A,B)$ is non-extendible, it
follows in view of Proposition \ref{AB_dual} that
$-A+\overline{A+B}=\overline{B}$, with $(-A,\overline{A+B})$
non-extendible. Since $A+B$ is aperiodic, it follows that $B$, and
thus $\overline{B}$, is aperiodic. Hence from the result of the
previous paragraph applied to $-A$ and $\overline{A+B}$, it follows
that $\overline{A+B}$ is not quasi-periodic and
$\langle-\gamma+\overline{A+B}\rangle=\langle A\rangle=G$, where
$\gamma\in\overline{A+B}$. If $A+B$ is quasi-periodic with
quasi-period $H_a$, then it follows, in view of $\overline{A+B}$ not
quasi-periodic, that $G=\langle-\gamma+\overline{A+B}\rangle\leq
H_a$. Hence $H_a=G$, contradicting that $A+B$ is aperiodic, and
completing the proof.
\end{proof}

Next, we prove the following simple proposition that will be needed
in several places.


\begin{pro}\label{matching_prop} Let $A$ and $B$ be finite subsets of an
abelian group, and let $C\subseteq A+B$. Suppose
$|A|,\,|B|,\,|C|\geq 2$. If $(a+B)\cap C$ and $(b+A)\cap C$ are both
nonempty for every $a\in A$ and $b\in B$, then either (a) there
exist distinct $a,\,a'\in A$, distinct $b,\,b'\in B$ and distinct
$c,\,c'\in C$ such that $a+b=c$ and $a'+b'=c'$, or else (b)
$|A|=|B|=|C|=2$, and there exists $a\in A$ and $b\in B$ such that
$a+B=b+A=C$.
\end{pro}

\begin{proof} Since $C\subseteq A+B$, let $a_1+b_1=c_1$ with
$a_1\in A$, $b_1\in B$ and $c_1\in C$. Since $|C|\geq 2$, let
$c_2\in C$ be distinct from $c_1$. Since $C\subseteq A+B$, it
follows that $c_2=a+b$ for some $a\in A$ and $b\in B$. If $a\neq
a_1$ and $b\neq b_1$, then the proof is complete. Otherwise, we may
w.l.o.g. by symmetry assume $c_2=a_2+b_1$, with $a_2\neq a_1$ (since
$c_2\neq c_1$). Since $|B|\geq 2$, let $b_2\in B$ with $b_2\neq
b_1$. Since $(b_2+A)\cap C$ is nonempty, it follows that $a+b_2=c$
for some $a\in A$ and $c\in C$. If $c\notin \{c_1,c_2\}$, then (a)
is satisfied with $a+b_2=c$ and either $a_1+b_1=c_1$ (if $a\neq
a_1$) or $a_2+b_1=c_2$ (if $a\neq a_2)$. Thus we can w.l.o.g. by
symmetry assume $c=c_2$. If $a\notin \{a_1,a_2\}$, then (a) is
satisfied with $a+b_2=c_2$ and $a_1+b_1=c_1$. Thus, since
$a_2+b_1=c_2$ (implying $a_2+b_2\neq c_2$), it follows that we can
assume $a=a_1$. Hence $a_1+b_2=c_2$, $a_1+\{b_1,b_2\}=\{c_1,c_2\}$
and $b_1+\{a_1,a_2\}=\{c_1,c_2\}$.

Suppose there exists $a\in A\setminus\{a_1,a_2\}$. Hence, since
$(a+B)\cap C$ is nonempty, it follows that $a+b=c$ for some $b\in B$
and $c\in C$. If $c\notin\{c_1,c_2\}$, then (a) is satisfied with
$a+b=c$ and either $a_1+b_1=c_1$ (if $b\neq b_1$) or $a_1+b_2=c_2$
(if $b\neq b_2)$. Therefore we can assume $c\in \{c_1,c_2\}$. If
$b\notin \{b_1,b_2\}$, then (a) is satisfied with $a+b=c$ and either
$a_1+b_1=c_1$ (if $c\neq c_1$) or $a_1+b_2=c_2$ (if $c\neq c_2)$.
Therefore, since $a_1+b_1=c_1$ and $a_2+b_1=c_2$ (implying
$a+b_1\notin\{c_1,c_2\}$), and since $a_1+b_2=c_2$ (implying
$a+b_2\neq c_2$), it follows that $b=b_2$ and $c=c_1$. Thus (a) is
satisfied with $a_2+b_1=c_2$ and $a+b_2=c_1$. So we can assume
$|A|=2$. Likewise, the same argument applied to $B$ shows that
$|B|=2$.

Suppose $c\in C\setminus\{c_1,c_2\}$. Since $C\subseteq A+B$, it
follows that $a+b=c$ for some $a\in A$ and $b\in B$. Since
$|A|=|B|=2$, and since $a_1+b_1=c_1$, $a_1+b_2=c_2$, and
$a_2+b_1=c_2$, it follows that $c=a_2+b_2$, whence (a) is satisfied
with $a_2+b_2=c$ and $a_1+b_1=c_1$. So we can assume $|C|=2$. Hence
in view of the conclusion of the first paragraph, it follows that
$a_1+B=a_1+\{b_1,b_2\}=\{c_1,c_2\}=C$ and
$b_1+A=b_1+\{a_1,a_2\}=\{c_1,c_2\}=C$, whence (b) holds.
\end{proof}


In the proof, once we have established that Theorem
\ref{KST_Step_Beyond} holds for $A'+B'=C'$, we will want to transfer
the resulting structure back to $A+B=C$. Since one of $A'$, $B'$ and
$C'$ will be a set from
$\pm\{A,B,A+B,\overline{A},\overline{B},\overline{A+B}\}$, our
strategy will be to use this common linking set (along with
Proposition \ref{AB_dual}) as the means of transferring the
structural information. However, to accomplish this, we will need to
know that having the unpaired structural information for only the
set $A$, is enough to conclude that Theorem \ref{KST_Step_Beyond}
holds for the pair that includes $A$. The following two lemmas
accomplish this in the case when a non-quasi-periodic, generating
subset $A$ is close to being quasi-periodic. Note that we have begun
assuming $G$ finite in the hypotheses of the lemmas. This is done to
make use of the set $\overline{A+B}$ via Proposition \ref{AB_dual}.
However, since we will be able to later show that the case $G$
infinite follows from the case $G$ finite, this will not hinder the
proof.

\begin{lem}\label{kemp_lem_sec_transfer} Let $A$ and $B$ be nonempty
subsets of a finite abelian group $G$ with $|A+B|=|A|+|B|$, $0\in
A\cap B$, $|A|,\,|B|,\,\db(A+B,\mathcal{P})\geq 3$, $(A,B)$
non-extendible, $\langle A\rangle =G$ and $A$ not quasi-periodic. If
$A=A'\cup A_2\cup A_1$ with $A'$ a nonempty periodic subset with
maximal period $H_a$, $A_1\subseteq a_1+H_a$, and $A_2\subseteq
a_2+H_a$, for some $a_i\in G$, then $\db(C,\mathcal{QP})=1$ for all
$C\in \{A,B,A+B,\overline{A},\overline{B},\overline{A+B}\}$ and
(\ref{Kemp_extendible}) holds.
\end{lem}

\begin{proof}
Let $B'$ be the maximal subset of $B$ that is $H_a$-periodic, and
let $B\setminus B'=B_{b_1}\cup\ldots\cup B_{b_l}$ be an $H_a$-coset
decomposition of $B\setminus B'$. From the maximality of $B'$ it
follows that no $B_{b_i}$ is $H_a$-periodic. From Lemma
\ref{kemp_lem_gen2} it follows that $\langle B\rangle =G$ and that
$B$ is not quasi-periodic. Hence $l\geq 2$, as otherwise $H_a=G$
implying $A=G$, which contradicts $A+B$ aperiodic. Since $A$ is not
quasi-periodic, it follows that $A_1$ and $A_2$ are both partially
filled $H_a$-cosets that are nonempty and disjoint modulo $H_a$.

If $|\phi_a(A')+\phi_a(B)|<|\phi_a(A')|+|\phi_a(B)|-1$, then from
Kneser's Theorem it follows that $\phi_a(A')+\phi_a(B)$ is periodic,
contradicting either the maximality of $H_a$ for $A'$ or the fact
that $A$ is non-extendible. Therefore
\begin{equation}\label{transfer2_lemma_1}|\phi_a(A')+\phi_a(B)|\geq |\phi_a(A')|+
|\phi_a(B)|-1.\end{equation}

Let $C'$ be the maximal subset of $A+B$ that is $H_a$-periodic, and
let $(A+B)\setminus C'=C_{c_1}\cup C_{c_2}\cup \ldots \cup C_{c_r}$
be an $H_a$-coset decomposition. In view of Lemma
\ref{kemp_lem_gen2} it follows that $A+B$ is not quasi-periodic.
Thus $r\geq 2$ (as $A'\neq \emptyset$ implies $C'\neq \emptyset$).
Note that $A'+B\subseteq C'$. Hence from (\ref{transfer2_lemma_1})
it follows that
\begin{multline}\label{kemp_secondtrans1} |A+B|\geq
(|\phi_a(A')|+|\phi_a(B')|+l-1)|H_a|+\Sum{i=1}{r}|C_{c_i}|\geq
\\ |A|+|B|-(|A_1|+|A_2|+\Sum{i=1}{l}|B_{b_i}|)+(l-1)|H_a|+
\Sum{i=1}{r}|C_{c_i}|.
\end{multline}

Since the $C_{c_i}$ are partially filled $H_a$-cosets, it follows
that $\phi_a(C\setminus C')\subseteq \phi_a(A_1\cup
A_2)+\phi_a(B\setminus B')$. From the non-extendibility of $A$, and
since each $A_i$ is a partially filled $H_a$-coset, it follows that
each $a_i$ has at least one $b_{j}$ such that
$\phi_a(a_i)+\phi_a(b_{j})\in \phi_a(C\setminus C')$. In view of the
non-extendibility of $B$, and since each $B_{b_i}$ is a partially
filled $H_a$-coset, it follows that same holds true for each $b_i$.
Consequently, it follows in view of Proposition \ref{matching_prop}
that either there exist distinct $b_{i_1}$ and $b_{i_2}$, and
distinct $c_{i'_1}$ and $c_{i'_2}$, such that
$\phi_a(a_1+b_{i_1})=\phi_a(c_{i'_1})$ and
$\phi_a(a_2+b_{i_2})=\phi_a(c_{i'_2})$, or else $l=r=2$ and w.l.o.g.
$\phi_a(a_1+\{b_1,b_2\})=\{\phi_a(c_1),\phi_a(c_2)\}$ and
$\phi_a(b_1+\{a_1,a_2\})=\{\phi_a(c_1),\phi_a(c_2)\}$. We handle
these cases separately.

\emph{Case 1:} First assume there exist distinct $b_{i_1}$ and
$b_{i_2}$, and distinct $c_{i'_1}$ and $c_{i'_2}$, such that
$\phi_a(a_1+b_{i_1})=\phi_a(c_{i'_1})$ and
$\phi_a(a_2+b_{i_2})=\phi_a(c_{i'_2})$. Since the $C_{c_i}$ are
partially filled, it follows from Kneser's Theorem that
$|A_1+B_{b_{i_1}}|\geq |A_1|+|B_{b_{i_1}}|-|H_{d_1}|+\rho_1$, for
some proper subgroup $H_{d_1}<H_a$, where $\rho_1$ is the number of
$H_{d_1}$-holes in $A_1$ and $B_{b_{i_1}}$. Likewise
$|A_2+B_{b_{i_2}}|\geq |A_2|+|B_{b_{i_2}}|-|H_{d_2}|+\rho_2$. Hence,
since $|H_{d_i}|\leq \frac{1}{2}|H_a|$, and since $|H_a|>|B_{b_i}|$
for all $i$, it follows in view of (\ref{kemp_secondtrans1}) that
\begin{multline}\label{kemp_secondtrans2} |A+B|\geq
|A|+|B|-(|A_1|+|A_2|+\Sum{i=1}{l}|B_{b_i}|)+(l-1)|H_a|+
\Sum{i=1}{r}|C_{c_i}|\geq \\
|A|+|B|-(|A_1|+|A_2|+|B_{i_1}|+|B_{i_2}|)+|H_a|+|A_1+B_{b_{i_1}}|+|A_2+B_{b_{i_2}}|\geq
\\ |A|+|B|+|H_a|-|H_{d_1}|-|H_{d_2}|+\rho_1+\rho_2\geq |A|+|B|.
\end{multline}
Hence, since $|A+B|=|A|+|B|$, it follows that $r=l=2$,
$\rho_1=\rho_2=0$, and $|H_{d_1}|=|H_{d_2}|=\frac{1}{2}|H_a|$, as
otherwise the above estimate will be strict.

Since $r=2$, then $|H_a|=2$ would imply that
$\db(A+B,\mathcal{P})\leq 2$, a contradiction. Therefore
$|H_{d_1}|=|H_{d_2}|=\frac{1}{2}|H_a|>1$. If $H_{d_1}\cap H_{d_2}$
is nontrivial, then in view of $\rho_1=\rho_2=0$, it follows that
$A$, and thus $A+B$, is periodic, a contradiction. Thus it follows
that $|H_a|\geq |H_{d_1}||H_{d_2}|=\frac{1}{4}|H_a|^2$. Hence, since
$|H_{d_1}|=|H_{d_2}|=\frac{1}{2}|H_a|>1$, it follows that $|H_a|=4$.
Thus $\db(A+B,A+B+H_{d_1})\leq |H_a|-|H_{d_2}|=2$, contradicting
that $\db(A+B,\mathcal{P})\geq 3$, and completing the case.

\emph{Case 2:} Next assume that $l=r=2$, and that w.l.o.g.
$\phi_a(a_1+\{b_1,b_2\})=\phi_a(b_1+\{a_1,a_2\})=\{\phi_a(c_1),\phi_a(c_2)\}$.
Hence, since each $C_{c_i}$ is a partially filled $H_a$-coset, it
follows in view of Proposition \ref{mult_result} that
\be\label{wowowo}|A_1|+|B_{b_2}|\leq |H_a|\;\mbox{ and
}\;|B_{b_1}|+|A_2|\leq |H_a|.\ee Since
$\phi_a(b_1+\{a_1,a_2\})=\{\phi_a(c_1),\phi_a(c_2)\}$, it follows
that $|C_{c_1}|+|C_{c_2}|\geq |A_2|+|B_{b_1}|$, and from Kneser's
Theorem applied to $B_{b_1}+A_1$, it follows that equality is
possible only if $C_{c_1}\cup C_{c_2}=B_{b_1}+(A_1\cup A_2)$ is
$H_b$-periodic with $\phi_b(A_1)=1$ and $H_b\leq H_a$. Thus, since
$A+B$ is aperiodic, it follows that equality is possible only if
$|A_1|=1$. The same argument applied to
$\phi_a(a_1+\{b_1,b_2\})=\{\phi_a(c_1),\phi_a(c_2)\}$ also shows
that $|C_{c_1}|+|C_{c_2}|\geq |B_{b_2}|+|A_1|$, with equality
possible only if $|B_{b_1}|=1$. Hence, since $|A+B|\leq |A|+|B|$, it
follows in view of (\ref{wowowo}) and (\ref{kemp_secondtrans1}),
that we must in fact have equality in both the estimates
$|C_{c_1}|+|C_{c_2}|\geq |B_{b_2}|+|A_1|$ and
$|C_{c_1}|+|C_{c_2}|\geq |A_2|+|B_{b_1}|$, as well as both
inequalities in (\ref{wowowo}). Consequently, $|A_1|=|B_{b_1}|=1$,
$|A_2|=|B_{b_2}|=|H_a|-1$, and $C_{c_1}\cup C_{c_2}=B_{b_1}+(A_1\cup
A_2)$. Thus in view of Lemma \ref{kemp_lem_gen2}, it follows that
$\db(C,\mathcal{QP}_{H_a})=\db(C,\mathcal{QP})=1$ for all $C\in
\{A,B,A+B,\overline{A},\overline{B},\overline{A+B}\}$. Hence, since
$\db(A+B,\mathcal{P})\geq 3$, it follows that $|H_a|\geq 3$. Thus,
since the $C_{c_i}$ are partially filled, and since $C_{c_1}\cup
C_{c_2}=B_{b_1}+(A_1\cup A_2)$, it follows in view of Proposition
\ref{mult_result} that $\phi_a(a_2+b_2)\neq \phi_a(c_i)$ for all
$i$, and that $\phi_a(a_2+b_1)=\phi_a(a_1+b_2)$. Consequently,
letting $\alpha$ be the $H_a$-hole in $A_2$, and letting $\beta$ be
the $H_a$-hole in $B_{b_2}$, it follows that
$|A\cup\{\alpha\}+B\cup\{\beta\}|=|A+B|+1=|A\cup\{\alpha\}|+|B\cup\{\beta\}|-1$,
yielding (\ref{Kemp_extendible}), and completing the proof.
\end{proof}


\begin{lem}\label{kemp_lem_punc_QP} Let $A$ and $B$ be nonempty
subsets of a finite abelian group $G$ with $|A+B|=|A|+|B|$, $0\in
A\cap B$, $|A|,\,|B|,\,\db(A+B,\mathcal{P})\geq 3$, $(A,B)$
non-extendible, and $\langle A\rangle =G$. If
$\db(A,\mathcal{QP})=1$ and either $|A|\geq 4$ or $|B|\geq 4$, then
$\db(C,\mathcal{QP})=1$ for all $C\in
\{A,B,A+B,\overline{A},\overline{B},\overline{A+B}\}$ and
(\ref{Kemp_extendible}) holds.
\end{lem}

\begin{proof}In view of Lemma \ref{kemp_lem_sec_transfer}, it
follows the proof is complete unless $A=A_1\cup A_0$ with each $A_i$
a subset of an $H_a$-coset and $|A_1|=|H_a|-1$, for some nontrivial
subgroup $H_a$. If $A_0$ is empty, then $\langle A\rangle =G$
implies that $H_a=G$, whence $|A|=|G|-1$, contradicting
$\db(A+B,\mathcal{P})\geq 3$. Therefore we can assume $A_0$ is
nonempty. Let $b_1,\ldots,b_c$ be those elements of $B$ that are the
unique element from their $H_a$-coset in $B$, and let
$B'=B_{b'_1}\cup\ldots \cup B_{b'_l}$ be an $H_a$-coset
decomposition of the remaining elements of $B$. Since $A$ is not
quasi-periodic it follows that $|A_0|<|H_a|$. Hence $|A|\geq 3$
implies $|H_a|\geq 3$. In view of Lemma \ref{kemp_lem_gen2}, it
follows that $B$ is not quasi-periodic and $\langle B\rangle=G$. We
divide the proof into two cases.

\emph{Case 1.} Suppose $|\phi_a(A)+\phi_a(B)|\geq
|\phi_a(A)|+|\phi_a(B)|-1=|\phi_a(B)|+1$. Hence in view of
Proposition \ref{mult_result} it follows that
\begin{multline*}|A+B|\geq
(l+c)|H_a|-c+|A_0|= \\
(|H_a|+|A_0|-1)+(l|H_a|+c)+c(|H_a|-2)-|H_a|+1\geq
|A|+|B|+\rho'+(c-1)(|H_a|-2)-1,\end{multline*} where $\rho'$ is the
number of $H_a$-holes in $B'$. Hence $|A+B|\leq |A|+|B|$ implies
\begin{equation}\label{kemp_2ndTrans-1}\rho'+(c-1)(|H_a|-2)-1\leq 0.\end{equation}
Note $c\geq 1$, since otherwise adding the $H_a$-hole contained in
$A_1$ will, in view of Proposition \ref{mult_result}, contradict
the non-extendibility of $A$.

Suppose $\rho'=0$. If $B'$ is empty, then $|B|\geq 3$ implies $c\geq
3$; otherwise, the cases $c\leq 2$ are covered by Lemma
\ref{kemp_lem_sec_transfer}. Thus in all cases we can assume $c\geq
3$, whence (\ref{kemp_2ndTrans-1}) implies $2|H_a|\leq 5$,
contradicting $|H_a|\geq 3$. So we can assume $\rho'>0$.

If $\rho'>1$, then (\ref{kemp_2ndTrans-1}) and $|H_a|\geq 3$ imply
$c\leq 0$, a contradiction. Therefore we can assume $\rho'=1$,
whence (\ref{kemp_2ndTrans-1}) and $|H_a|\geq 3$ imply $c=1$. Hence
if $|\phi_a(B)|>2$, then the proof is complete in view of $\rho'=1$
and Lemma \ref{kemp_lem_sec_transfer}. Otherwise $\rho'=c=1$ implies
that $B=B_1\cup B_0$ with both $B_i$ nonempty subsets of disjoint
$H_a$-cosets, $|B_1|=|H_a|-1$, and $|B_0|=1$. Thus the hypotheses of
the lemma and the case are satisfied by interchanging the roles of
$A$ and $B$, whence $|A_0|=1$ follows by the above argument as well.
Since $A$ is not quasi-periodic and since $\langle A\rangle =G$, it
follows that $\overline{A}$ is not quasi-periodic. Likewise for $B$.
Hence $\db(B,\mathcal{QP})=\db(\overline{B},\mathcal{QP})=
\db(A,\mathcal{QP})=\db(\overline{A},\mathcal{QP})=1$. In view of
Proposition \ref{mult_result} and $|H_a|\geq 3$, it follows that
$|A_1+B_1|=|H_a|$ and $A_0+B_1=A_1+B_0$, since otherwise $|A+B|\geq
2|H_a|+1>|A|+|B|$, a contradiction. Hence
$\db(A+B,\mathcal{QP}_{H_a})=\db(\overline{A+B},\mathcal{QP}_{H_a})=1$.
Thus, since neither $A+B$ nor $\overline{A+B}$ is quasi-periodic (in
view of Lemma \ref{kemp_lem_gen2}), it follows that
$\db(A+B,\mathcal{QP})=\db(\overline{A+B},\mathcal{QP})=1$. Finally,
by letting $\alpha$ be the $H_a$-hole in $A_1$, and by letting
$\beta$ be the $H_a$-hole in $B_1$, it follows in view of
$A_1+B_0=A_0+B_1$ and $|A_1+B_1|=|H_a|$ that
$|A\cup\{\alpha\}+B\cup\{\beta\}|=|A+B|+1=|A\cup\{\alpha\}|+|B\cup\{\beta\}|-1$,
yielding (\ref{Kemp_extendible}), and completing the case.

\emph{Case 2.} Suppose $|\phi_a(A)+\phi_a(B)|=|\phi_a(B)|$. Thus
from Kneser's Theorem it follows that $\phi_a(B)$ is periodic with
maximal subgroup $H_b/H_a$, and that $\phi_a(A)$ is contained in an
$H_b/H_a$-coset. Hence, since $\langle A\rangle =G$, it follows that
$H_b=G$ and that $G/H_a$ is cyclic generated by
$\phi_a(A_1)-\phi_a(A_0)$. Letting $\rho'$ be the number of
$H_a$-holes in $B'$, it follows in view of Proposition
\ref{mult_result} that
\begin{multline}\label{kemp_lem_trans8}|A+B|\geq (l+c)|H_a|-c=
(|H_a|-1+|A_0|)+(l|H_a|+c)+(c-1)(|H_a|-2)-|A_0|-1\geq \\
|A|+|B|-|A_0|+\rho'+(c-1)(|H_a|-2)-1,\end{multline} with equality
possible only if $A_1+b_i+H_a\nsubseteq A+B$ for $i=1,\ldots,c$.
Since $\db(A+B,\mathcal{P})\geq 3$, it follows, in view of
Proposition \ref{mult_result} and the assumption of the case, that
$c\geq 3$, with equality possible only if $A_1+b_i+H_a\nsubseteq
A+B$ for $i=1,\ldots,c$.

Suppose $l>0$. Hence, since $\phi_a(A_1)-\phi_a(A_0)$ generates
$G/H_a$, and since $c\geq 3>0$, it follows that there exists $b'_j$
such that $\phi_a(A_0+b'_j)=\phi_a(A_1+b_i)$ for some $b_i$. Thus it
follows in view of Proposition \ref{mult_result} that either
$\rho'\geq |A_0|$, or else $A_1+b_i+H_a= A_0+b'_j+H_a\subseteq A+B$.
In the former case, it follows in view of (\ref{kemp_lem_trans8}),
$|A+B|\leq |A|+|B|$ and $|H_a|\geq 3$, that $c\leq 2$, a
contradiction. In the latter case, it follows that the inequality
(\ref{kemp_lem_trans8}) is strict and that $c\geq 4$. Hence it
follows, in view of $|A+B|\leq |A|+|B|$ and $|A_0|\leq |H_a|-1$,
that $2|H_a|\leq 5$, contradicting $|H_a|\geq 3$. So we may assume
$l=0$, whence $c=|B|$.

Suppose $|A_0|\leq |H_a|-2$. In view of (\ref{kemp_lem_trans8}),
$|A+B|\leq |A|+|B|$, $|H_a|\geq 3$, and $c\geq 3$, it follows that
$|A_0|\geq 2|H_a|-5\geq |H_a|-2$. Thus $|A_0|\leq |H_a|-2$ implies
$|H_a|=3$, whence $|A_0|\leq |H_a|-2$ implies $|A|=3$. Thus by
hypothesis $c=|B|\geq 4$, whence (\ref{kemp_lem_trans8}) implies
$|A_0|\geq 3|H_a|-7\geq |H_a|-1$, a contradiction. So we can assume
$|A_0|=|H_a|-1$.

If $A_1+b_i+H_a\subseteq A+B$ for some $i$, then
(\ref{kemp_lem_trans8}) will be strict and $c\geq 4$, whence
$|A_0|\geq 3|H_a|-6\geq |H_a|$, a contradiction. Therefore we can
assume $A_1+b_i+H_a\nsubseteq A+B$ for all $i$. Thus, since
$\phi_a(B)$ is $G/H_a$-periodic with $G/H_a$ cyclic generated by
$\phi_a(A_1)-\phi_a(A_0)$, and since $|A_1|=|A_0|=|H_a|-1$, it
follows that we can permute the $b_i$ such that $A_1+b_i=A_0+b_{j}$
for $i\equiv j+1 \mod c$. Consequently,
\begin{multline*} A+B+\{b_1,b_{2}\}=\left(\bigcup_{i=1}^c
A_1+b_i\right)+\{b_1,b_{2}\}= \left(\bigcup_{i=1}^c
A_1+b_i+b_1\right)\cup \left(\bigcup_{i=1}^c A_1+b_{2}+b_i\right)=\\
\left(\bigcup_{i=1}^c A_1+b_i+b_1\right)\cup \left(\bigcup_{i=1}^c
A_0+b_{1}+b_i\right)=\left(\left(\bigcup_{i=1}^c A_1+b_i\right)\cup
\left(\bigcup_{i=1}^c
A_0+b_{i}\right)\right)+b_1=A+B+b_1,\end{multline*} implying from
Kneser's Theorem that $A+B$ is periodic, a contradiction.
\end{proof}


The following lemma will also be used to transfer unpaired
structural information from the single set $A$ to a pair containing
$A$, this time in the case when $A$ is an arithmetic progression.

\begin{lem}\label{kemp_lem_AP_transfer} Let $A$ and $B$ be finite, nonempty
subsets of an abelian group $G$ with $|A+B|=|A|+|B|$, $|A|\geq 3$,
and $A+B$ aperiodic.  If $A$ is an arithmetic progression with
difference $d$, then $h_d(B)=1$, $h_d(A+B)=h_d(\overline{A+B})=0$
and (\ref{Kemp_extendible}) holds.
\end{lem}

\begin{proof} Note that $|A+B|\geq |B|+c(|A|-1)$, where $c$ is the
number of $d$-components of $\overline{B}$ with length at least
$|A|-1$. Since $A+B$ is aperiodic, it follows that there must be at
least one $d$-component of $\overline{B}$ with length at least
$|A|-1$. Hence either $|A+B|\geq |B|+2(|A|-1)$ or else $|A+B|=
|B|+|A|-1+h_d(B)$. Since $|A|\geq 3$, and since $|A+B|=|A|+|B|$, it
follows that former cannot hold, whence the latter implies
$h_d(B)=1$. Hence, since $|A|\geq 3$, it follows that $h_d(A+B)=0$,
whence $h_d(\overline{A+B})=0$ as well. Letting $\beta$ be the
single hole in $B$ and letting $\alpha\in A$, it follows that
(\ref{Kemp_extendible}) holds.
\end{proof}





The following lemma will be one of our main tools for reducing the
case $A+B=C$ to a case $A'+B'=C'$. Lemma
\ref{kemp_lem_2-comp-transfer} will allow us to conclude the sets
$A'=A+\{0,d\}$, $B'=B$ and $C'=C+\{0,d\}$ also satisfy
$|A'+B'|=|A'|+|B'|$ (whence induction will be employed), provided
$c_d(A)=2$ for some nonzero $d$. We note this was (more or less) the
main strategy used to prove the prime order case of Theorem
\ref{KST_Step_Beyond} in \cite{ham-rodseth}. Lemma
\ref{kemp_lem_2-comp-transfer} will also be needed for Lemma
\ref{kemp_lem_AAP_transfer}.

\begin{lem}\label{kemp_lem_2-comp-transfer} Let $A$ and $B$ be nonempty subsets of a finite abelian
group $G$ with $|A+B|=|A|+|B|$, $0\in A\cap B$, $|A|\geq 4$,
$|B|\geq 3$, $\db(A+B,\mathcal{P})\geq 3$, $(A,B)$ non-extendible,
$\langle A\rangle =G$ and $A$ not quasi-periodic. If $c_d(A)=2$ for
some nonzero $d$, then either $c_d(B),\,c_d(A+B)\leq 2$, or else
(\ref{Kemp_extendible}) holds and either $\db(C,\mathcal{QP})=1$ for
all $C\in \{A,B,A+B,\overline{A},\overline{B},\overline{A+B}\}$, or
$\db(B,\mathcal{AP})=0$.
\end{lem}

\begin{proof} Since $A$ is non-extendible, it follows in view of Proposition \ref{AB_dual}
that $-B+\overline{A+B}=\overline{A}$. Suppose $c_d(B)\geq 3$. Hence
since $c_d(A)=c_d(\overline{A})=2$, and since $|A+B|=|A|+|B|$, it
follows that
\be\label{2comp-trans-1}|(-B+\{0,d\})+\overline{A+B}|\leq
|-B+\{0,d\}|+|\overline{A+B}|-1.\ee If $\overline{A}+\{0,d\}$ is
periodic, then $|A|\geq 3$ and $c_d(\overline{A})=2$ imply that $A$
is a union of a nonempty periodic set and at most two elements,
whence Lemma \ref{kemp_lem_sec_transfer} implies
$\db(A,\mathcal{QP})=1$. Thus Lemma \ref{kemp_lem_punc_QP} completes
the proof. Therefore we can assume that $\overline{A}+\{0,d\}$ is
aperiodic. Hence Kneser's Theorem and (\ref{2comp-trans-1}) imply
$c_d(B)=3$ and $|(-B+\{0,d\})+\overline{A+B}|=
|-B+\{0,d\})|+|\overline{A+B}|-1$, whence we can apply KST to the
pair $(-B+\{0,d\},\overline{A+B})$. Let $-B+\{0,d\}=B_1\cup B_0$ and
$\overline{A+B}=C_1\cup C_0$ be the Kemperman decompositions with
common quasi-period $H_a$.

In view of Lemma \ref{kemp_lem_gen2}, it follows that
$\overline{A+B}$ is not quasi-periodic and that $\langle
\gamma-\overline{A+B}\rangle=G$ for $\gamma\in A+B$. Hence $H_a=G$,
whence we cannot have type (I), nor as the sumset is aperiodic can
we have type (III).

Suppose we have type (II). Hence $\overline{A+B}$ is an arithmetic
progression with difference $d'$, whence Lemma
\ref{kemp_lem_AP_transfer} applied to $(-B,\overline{A+B})$ implies
$h_{d'}(B)=1$ and $h_{d'}(\overline{A})=h_{d'}(A)=0$. However, in
view of Proposition \ref{AB_dual}, and Lemma
\ref{kemp_lem_AP_transfer} applied to $(-A,\overline{A+B})$, it
follows that $h_{d'}(A)=1$, contradicting $h_{d'}(A)=0$. So we
cannot have type (II), and thus must have type (IV). Hence, since
$\langle -\gamma+\overline{A+B}\rangle =G$, for $\gamma\in
\overline{A+B}$, it follows that $|\overline{A}+\{0,d\}|=|G|-1$,
implying $|A|\leq 3$, a contradiction. So we can assume $c_d(B)\leq
2$.

Applying the above argument with the roles of $B$ and
$\overline{A+B}$ interchanged, it follows that either
$c_d(\overline{A+B})=c_d(A+B)\leq 2$, or else $B$ is an arithmetic
progression. Hence, since $c_d(B)\leq 2$, it follows that we can
assume the later case holds, else the proof is complete. Thus in
view of Lemma \ref{kemp_lem_AP_transfer} it follows that
(\ref{Kemp_extendible}) holds, completing the proof.
\end{proof}


The next lemma stretches Lemma \ref{kemp_lem_AP_transfer} one step
further, to handle the case $\db(A,\mathcal{AP})=1$.

\begin{lem}\label{kemp_lem_AAP_transfer} Let $A$ and $B$ be
nonempty subsets of a finite abelian group $G$ with $|A+B|=|A|+|B|$,
$0\in A\cap B$, $|A|,\,|B|,\,\db(A+B,\mathcal{P})\geq 3$, and
$(A,B)$ non-extendible, $\langle A\rangle =G$ and $A$ not
quasi-periodic. If $\db(A,\mathcal{AP}_d)=1$ for some nonzero $d\in
G$, and if at most one of $|A|$, $|B|$ and $|\overline{A+B}|$ is
equal to $3$, then (\ref{Kemp_extendible}) holds, and one of
$h_d(B)\leq 1$, or $h_{d'}(B)=0$ for some non-zero $d'$, or
$\db(C,\mathcal{QP})=1$ for all $C\in
\{A,B,A+B,\overline{A},\overline{B},\overline{A+B}\}$, also holds
\end{lem}

\begin{proof}
Since $\db(A,\mathcal{AP}_d)=1$, it follows that $A$ is a subset of
an arithmetic progression with difference $d$ and one hole. Hence,
$c_d(A)= 2$. Furthermore, since $\langle A\rangle =G$, it follows
that $\langle d\rangle=G$. Since $B$ is not quasi-periodic (in view
of Lemma \ref{kemp_lem_gen2}), it follows that $h_d(B)=0$ implies
$B$ is an arithmetic progression, whence Lemma
\ref{kemp_lem_AP_transfer} completes the proof. So we can assume
$h_d(B)>0$.

By translating and considering $-A$ and $-B$ if necessary, we may
w.l.o.g. assume $0,\,d\in A$, and that $0$ is the first term of the
minimal arithmetic progression with difference $d$ containing $A$.
Since $\langle d\rangle =G$, let $B_1,\ldots,B_c$ be the
$d$-components of $B$ cyclicly ordered according to the direction
given by $d$. Let $B'_i$ be the $d$-component of $\overline{B}$
located between $B_i$ and $B_{i+1}$, with indices taken modulo $c$.
Hence $G$ is the disjoint union of the $B_i$ and $B'_i$. Observe,
since $0$ and $d$ are the first two terms in $A$, that for each $i$
either at least $\min\{|A|,|B'_i|\}$ of the holes contained in
$B'_i$ are elements of $A+B$, or else $|B_i|=1$ and at least
$\max\{1,\,\min\{|A|-1,|B'_i|-1\}\}$ of the holes contained in
$B'_i$ are elements of $A+B$. Hence, since $|B|+2(|A|-1)>|A|+|B|$,
it follows that $|B'_i|<|A|$ for all but at most one $i$ (say $c$).

If $|A|=3$, then $|A+B|=|B|+3$ implies $c\leq 3$. Hence if
$|B'_i|<|A|$ for all $i$, then $|A+B|\geq |\langle A\rangle|-c\geq
|G|-3$. Thus $|A|=|\overline{A+B}|=3$, a contradiction. If $|A|\geq
4$, then in view of Lemma \ref{kemp_lem_2-comp-transfer} it follows
that $c\leq 2$, else the proof is complete. Hence if $|B'_i|<|A|$
for all $i$, then $|A+B|\geq |\langle A\rangle|-c\geq |G|-2$,
contradicting $\db(A+B,\mathcal{P})\geq 3$. Thus regardless we can
assume $|B'_c|\geq |A|$. Consequently, the discussion of the
previous paragraph implies \be\label{hohomerryxmas}|A+B|\geq
|A|+|B|+h_d(B)-x+y,\ee where $x$ is the number of $B_i$ with
$|B_i|=1$, and $y$ is the number of $B'_i$ with $|B_i|=|B'_i|=1$.

Suppose $|A|=3$. Hence $A=\{0,d,3d\}$, and as noted in the previous
paragraph, $c\leq 3$. Since $|B|\geq 4$, it follows that $x\leq
c-1\leq 2$. On the other hand, $h_d(B)\geq c-1$, with equality
possible only if $|B'_i|=1$ for $i\leq c-1$. Thus from
(\ref{hohomerryxmas}) it follows that  $|A|+|B|=|A+B|\geq
|A|+|B|+h_d(B)-x+y\geq |A|+|B|$, whence indeed $|B'_i|=1$ for $i\leq
c-1$ and $y=0$. Hence if $c=2$, then $h_d(B)=1$, whence letting
$\alpha$ be the hole in $A$ and letting $\beta$ be the hole in $B$
yields (\ref{Kemp_extendible}) (recall that $\langle d\rangle =G$).
If $c=1$, then $h_d(B)=0$, a contradiction to the conclusion of the
first paragraph. Therefore assume $c=3$. Hence, since each $B'_i$
contributes at least one to the sumset, and since $B'_3$ contributes
at least $|A|-1$ (in view of the discussion in the second paragraph
of the proof), it follows that $|A+B|\geq |B|+|A|-1+(c-1)>|A|+|B|$,
again a contradiction. So we may assume $|A|\geq 4$, and (as noted
before) that $c\leq 2$.

Since $|B|\geq 3$ and since $c\leq 2$, it follows that at most one
$B_i$ can have cardinality one. Hence from (\ref{hohomerryxmas}) it
follows that $|A+B|\geq |A|+|B|+h_d(B)-1$, whence $h_d(B)=1$.
Letting $\alpha$ be the hole in $A$, and letting $\beta$ be the hole
in $B$ yields (\ref{Kemp_extendible}), and completes the proof.
\end{proof}


We conclude the list of lemmas with a short proof of a special case
of the Fainting Lemma from \cite{ham-ser-chowla}. We remark that the
idea for the proof of Lemma \ref{kemp_secondlemma} could be used to
prove a weaker form of the Fainting Lemma that does not require the
assumption about the first isoperimetric number.

\begin{lem}\label{kemp_secondlemma}
Let $A$ and $B$ be finite, nonempty subsets of an abelian group $G$
with $0\in A\cap B$, $|A|=3$, $|A+B|=|A|+|B|+m$, and $\langle
A\rangle=G$. If $|\nu_c(A,B)|\geq 2$ for all $c\in A+B$, then $G$ is
finite and $|B|\geq |G|-\binom{m+4}{2}$.\end{lem}

\begin{proof} In view of Proposition \ref{surjectivity_transfer},
it follows that $|\nu_c(B+(i-1)A,A)|\geq 2$ for all $c\in A+iB$ and
$i\geq 1$. Hence, since $|A|=3$, it follows that
$N_i^{A^*}=N_i^{\leq A^*}=N_i$ for all $i\geq 1$, where
$A^*=A\setminus 0$. Thus in view of Proposition
\ref{magic_N_i_lemma} it follows that
\be\label{isolem1}N_i-A^*\subseteq N_{i-1},\ee for all $i\geq 2$.
Note since $\langle A\rangle =G$, that either $B+lA=G$ for
sufficiently large $l$, if $G$ is finite, or else
$|B+lA|>|B+(l-1)A|$ for all $l$ (in view of Kneser's Theorem), if
$G$ is infinite. Thus if we can show that $|N_i|<|N_{i-1}|$ for
nonempty $N_{i-1}$ with $i\geq 2$, it will follow that
$N_i=\emptyset$ for sufficiently large $i$, whence $G$ is finite,
and that $|B|=|G|-\Summ{i\geq 1} |N_i|\geq
|G|-\Sum{i=0}{m+2}(m+3-i)=|G|-\binom{m+4}{2}$, completing the proof.
However, if $|N_i|\geq |N_{i-1}|>0$, then in view of (\ref{isolem1})
and Kneser's Theorem, it follows that $|N_i|=|N_{i-1}|$, that $N_i$
is periodic with maximal period $H_a$, and that $A^*$ is a subset of
an $H_a$-coset. Since $|\nu_c(B+(i-1)A,A)|\geq 2$ for all $c\in
B+iA$, it follows that $B+iA=B+(i-1)A+A^*$. Thus, since
$|\phi_a(A^*)|=1$, it follows that $\phi_a(B+(i-1)A)=\phi_a(B+iA)$,
whence $N_i=(B+iA)\setminus (B+(i-1)A)$ cannot be $H_a$-periodic, a
contradiction. Therefore $|N_i|<|N_{i-1}|$, completing the proof.
\end{proof}


We are now ready to proceed with the proof of Theorem
\ref{KST_Step_Beyond}. However, before beginning, we sketch the main
points to outline the strategy. We begin by handling the case $A+B$
periodic, and then show that we can restrict our attention to the
case when neither $A$ nor $B$ is quasi-periodic. We then handle the
cases when $\db(A+B,\mathcal{P})$ is small. The assumption that
$\db(A+B,\mathcal{P})$ is small will allow us (in most instances) to
show (\ref{Kemp_extendible}) fairly easily, since adding $H$-holes
to $A$ or $B$ can only increase $A+B$ by at most
$\db(A+B,\mathcal{P}_H)$ elements. However, there will be one
difficult instance that will instead lead to the type (VIII) pair.
Once we have established that $\db(A+B,\mathcal{P})\geq 3$, we can
restrict our attention to generating subsets and begin to gain
access to the lemmas we have just proved. To gain full access, we
must handle the case when $|A|\leq 3$. The case $|A|=|B|=3$ is
handled by brute force. We then restrict our attention to the case
$G$ finite. For the case $|A|=3$ with $|B|\geq 4$, we use Lemma
\ref{kemp_secondlemma} to show the existence of a unique expression
element $a+b$, and proceed by inductive arguments used on the pair
$(A,B\setminus b)$. These will fail if $A\subseteq B$ and $|B|=4$,
in which case an additional argument is used. With the cases
$\min\{|A|,|B|\}\leq 3$ complete, the proof then continues, for $G$
finite, by induction (assuming the theorem true for $A'$ and $B'$
with $\min\{|A'|,|B'|\}<\min\{|A|,|B|\}$ or
$\min\{|A'|,|B'|\}=\min\{|A|,|B|\}$ and $|A'|+|B'|>|A|+|B|$). We
employ the previously mentioned Dyson $e$-transform as the method to
obtain the pairs $A'+B'=C'$. The arguments from the proof of
Kneser's Theorem in \cite{PhD-Dissertation} will be extended to
handle the case when $A(e)+B(e)$ is periodic. The cases $|B(e)|\geq
3$ are handled by applying the induction hypothesis to the pair
$(A(e),B(e))$. This method fails when $|B(e)|\leq 2$, since in these
cases the unpaired structural information gained for a single set is
insufficient to directly transfer back to the original pair. In the
case $|B(e)|=1$, we instead use the method developed in Section 3.
In the case $|B(e)|=2$, then via Proposition \ref{AB_dual} we will
obtain $c_d(A)\leq 2$ for some nonzero $d$, whence we instead
consider $A+\{0,d\}+B$, as discussed before Lemma
\ref{kemp_lem_2-comp-transfer}. We will encounter problems if
$|\overline{A+B}|$ is small. The remaining cases will then be shown
to follow from the case $|A|=|B|=4$ with $c_d(A)=c_d(B)=c_d(A+B)=2$.
The proof with $G$ finite concludes by completing this last
remaining case directly. The case when $G$ is infinite is then
derived from the finite case by the use of an appropriate
\emph{Freiman isomorphism} of $(A,B)$ (which is an injective map
$\varphi:A\cup B\rightarrow G'$, with $G'$ an abelian group, such
that $\varphi(a_1)+\varphi(b_1)=\varphi(a_2)+\varphi(b_2)$ holds,
where $a_i\in A$ and $b_i\in B$, if and only if $a_1+b_1=a_2+b_2$).

\begin{proof}
We may assume w.l.o.g. that $0\in A\cap B$. If either $A$ or $B$ is
extendible, then (\ref{Kemp_extendible}) immediately follows.
Therefore we can assume otherwise, whence Proposition \ref{AB_dual}
implies $$-A+\overline{A+B}=\overline{B}\;\mbox{  and }
-B+\overline{A+B}=\overline{A}.$$

Suppose that $A+B$ is periodic with maximal period $H_a$. If $A$ and
$B$ are not both $H_a$-periodic, then w.l.o.g. there exists
$\alpha\in \overline{A}$ such that
$\phi_a(A)=\phi_a(A\cup\{\alpha\})$. Hence, since $A+B$ is
$H_a$-periodic, it follows that $A\cup\{\alpha\}+B=A+B$,
contradicting that $A$ is non-extendible.  Therefore we may assume
that $A$ and $B$ are both $H_a$-periodic. From Kneser's Theorem it
follows that
$$|\phi_a(A+B)|\geq |\phi_a(A)|+|\phi_a(B)|-1.$$ If equality holds
in the above inequality, then since $A$ and $B$ are both
$H_a$-periodic, it follows that
$$|A+B|=|H_a||\phi_a(A+B)|=|H_a+A|+|H_a+B|-|H_a|=|A|+|B|-|H_a|<|A|+|B|,$$
a contradiction. If $|\phi_a(A+B)|=|\phi_a(A)|+|\phi_a(B)|$, then
the proof is complete. Otherwise
$$|A+B|=|H_a||\phi_a(A+B)|\geq
|H_a+A|+|H_a+B|+|H_a|=|A|+|B|+|H_a|>|A|+|B|,$$ a contradiction once
more. So we may assume that $A+B$ is aperiodic.

Next suppose, for some $\gamma\in \overline{A+B}$, that
$A+B\cup\{\gamma\}$ is periodic with maximal period $H_a$. Note
$\phi_a(\gamma)\in \phi_a(A+B)$. Hence choosing $\alpha\in
(\gamma-B)\cap (H_a+A)$ and $\beta\in (\gamma-A)\cap
(H_a+B)$, it follows that
$A\cup\{\alpha\}+B\cup\{\beta\}=A+B\cup\{\gamma\}$, whence
(\ref{Kemp_extendible}) holds. So we may assume $\db(A+B,\mathcal{P})\geq 2$.

If $|A|=1$, then $|A+B|=|A|+|B|$ cannot hold, and if $|A|=2$, then
the theorem holds with type (V) and group $G$. So we can assume
$|A|,\,|B|\geq 3$.

It is readily checked that the theorem holding for $A_0$ and $B_0$
in Lemma \ref{kemp_Lemma_qp} implies that the theorem holds for $A$
and $B$. Thus, since the case $A+B$ periodic is complete, it follows
in view of Lemma \ref{kemp_Lemma_qp} (by considering reduced
quasi-periodic decompositions) that it suffices to prove the theorem
when neither $A$ nor $B$ is quasi-periodic. Thus we henceforth
assume this is the case. Hence, since $|A|,\,|B|\geq 3$, it follows
in view of Lemma \ref{kemp_generate_lem} that w.l.o.g. we may assume
$\langle A\rangle =\langle B\rangle=G$.

Suppose, for some distinct $\gamma_1,\gamma_2\in \overline{A+B}$,
that $A+B\cup\{\gamma_1,\gamma_2\}$ is periodic with maximal period
$H_a$. Note that $\phi_a(\gamma_i)\in \phi_a(A+B)$, for $i=1,2$,
since otherwise $A+B$ is periodic, a contradiction. Hence choosing
$\alpha_1\in (\gamma_1-B)\cap (H_a+A)$, and $\beta_1\in
(\gamma_2-A)\cap (H_a+B)$, it follows that
$A\cup\{\alpha_1\}+B\cup\{\beta_1\}=A+B\cup\{\gamma_1,\gamma_2\}$.
Since $\db(A+B,\mathcal{P})\geq 2$, it follows that either $A$ or
$B$, w.l.o.g. $A$, contains at least two $H_a$-holes. Thus we can
find $\alpha_2\in (H_a+A)\cap \overline{A\cup \{\alpha\}}$ such that
$A\cup\{\alpha_1,\alpha_2\}+B\cup\{\beta_1\}=A+B\cup\{\gamma_1,\gamma_2\}$
is maximally $H_a$-periodic. Hence from Kneser's Theorem it follows
that there are $\rho=|H_a|-1+3=|H_a|+2$ holes contained among the
sets $A$ and $B$, and that
\begin{equation}\label{kemp_4}|\phi_a(A+B)|= |\phi_a(A)|+|\phi_a(B)|-1.\end{equation}

Let
$(\phi_a(a_1),\phi_a(b_1)),(\phi_a(a_2),\phi_a(b_2)),\ldots,(\phi_a(a_l),\phi_a(b_l))
\in \phi_a(A)\times \phi_a(B)$, with $a_i\in A$ and $b_i\in B$, be
those pairs from $\phi_a(A)\times \phi_a(B)$ such that
$\phi_a(a_i+b_i)\in\{\phi_a(\gamma_1),\phi_a(\gamma_2)\}$. Since
$\gamma_i\notin A+B$, it follows in view of Proposition
\ref{mult_result} that
\begin{equation}\label{kemp_5}|A_{a_i}|+|B_{b_i}|\leq |H_a|,\end{equation} for all $i$.

Suppose $|\{\phi_a(a_i)\}_{i=1}^l|=1$. Hence in view of the
non-extendibility of $A$, it follows that $a+H_a\subseteq A$ for all
$a\in A\setminus A_{a_1}$. Thus, since $A$ is not-quasi-periodic, it
follows that $A=A_{a_1}$, whence $\langle A\rangle =G$ implies
$H_a=G$. Hence, since there are exactly $|G|-|B|$ elements
$\alpha\in G$ such that $\gamma_1\notin \alpha+B$, and since
$|A|+|B|=|G|-2$, it follows that there exists such an
$\alpha\in\overline{A}$. Likewise, since there are exactly
$|G|-|A|-1$ elements $\beta\in G$ such that $\gamma_1\notin
\beta+(A\cup \{\alpha\})$, and since $|A|+|B|=|G|-2$, it follows
that there exists such a $\beta\in \overline{B}$. Hence
$A\cup\{\alpha\}+B\cup\{\beta\}\subseteq G\setminus \gamma_1$,
whence the non-extendibility of $(A,B)$ implies equality, yielding
(\ref{Kemp_extendible}). So we can assume
$|\{\phi_a(a_i)\}_{i=1}^l|\geq 2$. By the same argument, it also
follows that $|\{\phi_a(b_i)\}_{i=1}^l|\geq 2$.

Hence in view of (\ref{kemp_5}), it follows that $\rho\geq 2|H_a|$,
with equality possible only if $|\{\phi_a(a_i)\}_{i=1}^l|=2$ and
$|\{\phi_a(b_i)\}_{i=1}^l|=2$. Thus in view of $\rho=|H_a|+2$, it
follows that $|H_a|=2$, that $|\{\phi_a(a_i)\}_{i=1}^l|=2$ and that
$|\{\phi_a(b_i)\}_{i=1}^l|=2$, implying $\phi_a(\gamma_1)\neq
\phi_a(\gamma_2)$ (else $A+B$ is periodic). There are three cases
for $l$.

Suppose $l=2$. Hence w.l.o.g.
$\phi_a(A_{a_1})+\phi_a(B_{b_1})=\phi_a(\gamma_1)$ and
$\phi_a(A_{a_2})+ \phi_a(B_{b_2})=\phi_a(\gamma_2)$ are both unique
modulo $H_a$ expression elements, whence letting $\alpha$ be the
other element from the $H_a$-coset $a_1+H_a$, and letting $\beta$ be
the other element from the $H_a$-coset $b_1+H_a$, it follows that
$A\cup\{\alpha\}+B\cup\{\beta\}=A+B\cup\{\gamma_1\}$, yielding
(\ref{Kemp_extendible}). So we can assume $l>2$

Suppose $l=3$. Hence it follows that some $\phi_a(a_i)$, say
$\phi_a(a_{j_1})$, is contained in only one pair
$(\phi_a(a_i),\phi_a(b_i))$. Likewise some $\phi_a(b_i)$, say
$\phi_a(b_{j_2})$, is also only contained in one pair
$(\phi_a(a_i),\phi_a(b_i))$. Hence, since $l=3$, it follows that
$b_{j_1}\neq b_{j_2}$ and $a_{j_1}\neq a_{j_2}$. Thus neither
$\phi_a(a_{j_2}+b_{j_2})$ nor $\phi_a(a_{j_1}+b_{j_1})$ can equal
$\phi_a(a_{j_2}+b_{j_1})\in\{\phi_a(\gamma_1),\phi_a(\gamma_2)\}$,
whence $\phi_a(a_{j_1}+b_{j_1})=\phi_a(a_{j_2}+b_{j_2})$, and
w.l.o.g. assume $\phi_a(a_{j_1}+b_{j_1})=\phi_a(\gamma_1)$. Hence,
letting $\alpha\in (\gamma_1-B)\cap (A_{a_{j_1}}+H_a)$ and letting
$\beta\in (\gamma_1-A)\cap (B_{b_{j_2}}+H_a)$, it follows that
$A\cup\{\alpha\}+B\cup \{\beta\}=A+B\cup\{\gamma_1\}$, whence
(\ref{Kemp_extendible}) holds. So we can assume $l=4$.

Let $A+B=C$. Hence, since each $C_{\gamma_i}$ is a partially filled
$H_a$-coset with $|H_a|=2$, it follows in view of $l=4$ that
$$C_{\gamma_1}\cup C_{\gamma_2}=(A_{a_1}\cup A_{a_2})+(B_{b_1}\cup B_{b_2})=b_1+(A_{a_1}\cup
A_{a_2})=a_1+(B_{b_1}\cup B_{b_2}),$$ implying from Kneser's Theorem
that $(A_{a_1}\cup A_{a_2})+(B_{b_1}\cup B_{b_2})=b_1+(A_{a_1}\cup
A_{a_2})=a_1+(B_{b_1}\cup B_{b_2})$ is periodic with maximal period
$H_{a'}$. Since $|B_{b_1}\cup B_{b_2}|=2$, it follows that
$|H_{a'}|=2$. Let $H_b=H_a\times H_{a'}\cong \Z/2\Z\times \Z/2\Z$.
Since $C_{\gamma_1}\cup C_{\gamma_2}$ is an $H_{a'}$-coset it
follows that it does not contain an $H_a$-coset. Hence
$\phi_b(\gamma_1)=\phi_b(a_1)+\phi_b(b_1)$ must be a unique
expression element in $\phi_b(A+B)$.

In view of the maximality of $H_a$, it follows that $\phi_a(A+B)$ is
aperiodic. Hence, in view of (\ref{kemp_4}), it follows that we can
apply KST to the pair $(\phi_a(A),\phi_a(B))$. Let
$\phi_{a}(A'_1)\cup \phi_{a}(A'_0)$ and $\phi_{a}(B'_1)\cup
\phi_{a}(B'_0)$ be the Kemperman decompositions with quasi-period
$H_{d}/H_{a}$, where $A=A'_1\cup A'_0$ and $B=B'_1\cup B'_0$. Since
$\phi_b(\{a_1,a_2\}+\{b_1,b_2\})$ is a unique expression element, it
follows that both elements in $\phi_{a}(\{a_1,a_2\}+\{b_1,b_2\})$
have exactly two representations in $\phi_{a}(A)+\phi_{a}(B)$ given
by
\be\label{two_rep}\phi_{a}(a_1)+\phi_{a}(b_2)=\phi_{a}(a_2)+\phi_{a}(b_1)\mbox{
and } \phi_{a}(a_1)+\phi_{a}(b_1)=\phi_{a}(a_2)+\phi_{a}(b_2).\ee

Since $\phi_{a}(A+B)$ is aperiodic, it follows that we cannot have
type (III). Suppose we have type (II). If say $a_1\in A'_1$, then
(since type (II) implies $|A_x|,\,|B_y|\geq 2$ for all $x\in A$ and
$y\in B$), it follows in view of (\ref{two_rep}) that
$\phi_d(a_1)=\phi_d(a_2)$ and that $\phi_d(b_1)=\phi_d(b_2)$.
Furthermore, if $b_i\in B'_1$, then $|H_d/H_a|=2$, and if $b_i\in
B'_0$, then $|\phi_a(B'_0)|=2$. However, since $\phi_a(B'_0)$ is an
arithmetic progression, and since $\{\phi_a(b_1),\phi_a(b_2)\}$ is
periodic with period $H_b/H_a$, it follows that $|H_d/H_a|=2$ in
this latter case as well. However, $|H_d/H_a|=2$ is impossible for
type (II) (since the order of the difference of the arithmetic
progression given by KST is at least $|A'_0|+|A'_1|-1\geq 3$).
Therefore, we can assume $a_1,a_2\in A'_0$ and $b_1,b_2\in B'_0$.
Hence, since $\phi_{a}(\{a_1,a_2\})$ is an $H_b/H_{a}$-coset with
$|H_b/H_{a}|=2$, and since $\phi_{a}(A'_0)$ is an arithmetic
progression whose difference generates the cyclic group
$\langle\phi_{a}(A'_0)-\phi_a(a_1)\rangle=\langle\phi_{a}(B'_0)-\phi_a(b_1)\rangle\leq
H_d/H_a$, it follows that
$|\phi_{a}(A'_0)|>\frac{|\langle\phi_{a}(A'_0)-\phi_a(a_1)\rangle|}{2}$.
Likewise
$|\phi_{a}(B'_0)|>\frac{|\langle\phi_{a}(B'_0)-\phi_a(b_1)\rangle|}{2}=
\frac{|\langle\phi_{a}(A'_0)-\phi_a(a_1)\rangle|}{2}$, whence
Proposition \ref{mult_result} implies that $\phi_{a}(A+B)$ is
$\langle\phi_{a}(A'_0)-\phi_a(a_1)\rangle$-periodic, a
contradiction. So type (II) cannot occur.

Suppose we have type (I) with w.l.o.g. $|\phi_{a}(A'_0)|=1$. Hence
some $a_i$, say $a_2$, is contained in $A'_1$. Thus, if some $b_i$
is contained in $B'_0$, then it follows in view of (\ref{two_rep})
that either $|\phi_{a}(B'_0)|=2$, $\phi_d(a_1)=\phi_d(a_2)$ and
$b_1,\,b_2\in B'_0$ (if $|\phi_a(B'_0)|\geq 2$), or else
$\phi_a(A'_0)=\{\phi_a(a_1)\}$ and w.l.o.g.
$\phi_a(B'_0)=\{\phi_a(b_2)\}$ (if $|\phi_a(B'_0)|=1$). In the
former case $\phi_{a}(B'_0)$, and thus $\phi_{a}(B)$, is
$H_b/H_{a}$-periodic, contradicting that $\phi_{a}(A+B)$ is
aperiodic. In the later case,
$\phi_a(a_1)+\phi_a(b_2)=\phi_a(a_2)+\phi_a(b_1)$ contradicts that
$\phi_a(A'_0)+\phi_a(B'_0)=\phi_a(a_1)+\phi_a(b_2)$ is a unique
expression element. Therefore we can assume $b_1,\,b_2\in B'_1$.
Thus, since $a_2\in A'_1$, it follows in view of (\ref{two_rep})
that $\phi_d(a_1)=\phi_d(a_2)$, $\phi_d(b_1)=\phi_d(b_2)$, and
$|H_d/H_a|=2$. Since $\phi_d(a_1)=\phi_d(a_2)$, it follows that
$H_b/H_a\leq H_d/H_a$. Hence $|H_d/H_a|=2$ implies that $H_d=H_b$.
Since $|H_d/H_a|=2$, and since $\phi_a(A+B)$ is aperiodic, it
follows that $|\phi_{a}(A'_0)|=|\phi_{a}(B'_0)|=1$. Since $H_d=H_b$,
and since $\phi_b(a_1)+\phi_b(b_1)$ is a unique expression element,
it follows that $\phi_d(a_1)+\phi_d(b_1)$ is a unique expression
element in addition to the unique expression element
$\phi_d(A'_0)+\phi_d(B'_0)$. Hence, since $a_1\notin A'_0$ and since
$b_1\notin B'_0$, then applying KST modulo $H_d=H_b$, it follows
that we must have type (II) with (by appropriate choice of sign)
both $\phi_d(a_i)$ and $\phi_d(b_i)$ the first term in their
respective arithmetic progression, and both $\phi_d(A'_0)$ and
$\phi_d(B'_0)$ the last term in their respective arithmetic
progression. Thus the theorem holds with type (VIII).

Finally, suppose we have type (IV). Hence $|H_d/H_a|\geq 6$,
$|\phi_a(A'_0)|,\,|\phi_a(B'_0)|\geq 3$, and $\phi_a(A'_0)$ is
aperiodic (all consequences for a type (IV) pair). Thus it follows,
by the same argument used in the case of type (II), that
$a_1,\,a_2\in A'_0$ and $b_1,\,b_2\in B'_0$. Hence $H_b<H_d$. Thus,
since $\phi_b(a_i)+\phi_b(b_i)$ is a unique expression element, it
follows in view of Proposition \ref{mult_result} that
\be\label{whoof}|\phi_b(A'_0+B'_0)|\geq
|\phi_b(A'_0)|+|\phi_b(B'_0)|-1.\ee In view of the description of
type (IV), it follows that $|\phi_b(A'_0)|=l_1+l_p$ and
$|\phi_b(B'_0)|=l_2+l_p$, where $l_p$ is the number of partially
filled $H_b/H_a$-cosets in $\phi_a(A'_0)$, which is also equal to
the number of partially filled $H_b/H_a$-cosets in $\phi_a(B'_0)$,
where $l_1$ is the number of $\phi_a(x)\in \phi_a(A'_0)$ with
$\phi_a(x)+H_b/H_a\subseteq \phi_a(A'_0)$, and where
$l_1+l_2+l_p=|H_d/H_b|$. Thus
$|\phi_b(A'_0)|+|\phi_b(B'_0)|-1=|H_d/H_b|+l_p-1$. Hence, since
$|\phi_b(A'_0+B'_0)|\leq |H_d/H_b|$ holds trivially, and since
$\phi_a(A'_0)$ is aperiodic, it follows in view of (\ref{whoof})
that $l_p=1$. Let $\phi_a(x)\in \phi_a(A'_0)$ and $\phi_a(y)\in
\phi_a(B'_0)$ be the elements that correspond to the unique
partially filled $H_b/H_a$-coset in $\phi_a(A'_0)$ and
$\phi_a(B'_0)$, respectively. Hence, since $\phi_a(A+B)$ is
aperiodic, and since $l_p=1$, it follows that $\phi_b(x)+\phi_b(y)$
is a unique expression element in $\phi_b(A'_0+B'_0)$. However,
since $|H_b/H_a|=2$, it follows that there is only one element from
the coset $x+H_b/H_a$ contained in $\phi_a(A'_0)$, and likewise for
the coset $y+H_b/H_a$ in $\phi_a(B'_0)$, whence
$\phi_a(x)+\phi_a(y)$ is a unique expression element, contradicting
that there are no unique expression elements in a type (IV) pair,
and completing the case when $\db(A+B,\mathcal{P})\leq 2$. So we can
assume $\db(A+B,\mathcal{P})\geq 3$.

Suppose that $|A|=|B|=3$. If $(A-A)\cap (B-B)=\{0\}$, then
$|A+B|=|A||B|=9>6=|A|+|B|$, a contradiction. Thus w.l.o.g.
$A=\{0,d,a_1\}$ and $B=\{0,d,a_2\}$. In view of Lemma
\ref{kemp_lem_AP_transfer}, it follows that we can assume neither
$A$ nor $B$ is an arithmetic progression, else the proof is
complete. Since $A$ is not quasi-periodic, it follows that no two
elements from $A$ can form a coset of an order two subgroup.
Likewise for $B$.

Note
\begin{equation}\label{kemp_equals3}A+B=\{0,d,2d,a_1,a_1+d,a_2,a_2+d,a_1+a_2\}.\end{equation}
If $a_1=a_2$, then the theorem follows with type (VI). If
$|\{0,d,2d\}|\leq 2$, then $\{0,d\}$ is a subgroup of order $2$,
which we have noted is not the case. Therefore $|\{0,d,2d\}|=3$ and
$2d\neq 0$. Since neither $A$ nor $B$ is an arithmetic progression,
it follows that $a_1,a_2\notin \{2d,-d\}$. Hence, since $a_1+d=a_2$
and $a_2+d=a_1$ together contradict that $2d\neq 0$, it follows that
$$|\{0,d,2d,a_1,a_1+d,a_2,a_2+d\}|\geq
|\{0,d,2d\}|+|\{a_1,a_1+d,a_2,a_2+d\}|\geq 6,$$ with equality
possible only if w.l.o.g. $a_2=a_1+d$. Thus in view of
(\ref{kemp_equals3}) and $|A+B|=6$, it follows that
$$a_1+a_2=2a_1+d\in A+B=\{0,d,2d,a_1,a_1+d,a_1+2d\},$$ implying that
one of the following hold: $2a_1+d=0$, $2a_1=0$, $2a_1=d$, $a_1=-d$,
$a_1=0$, or $a_1=d$. The last three equalities are contradictions.
If $2a_1=0$, then $\{0,a_1\}$ is a coset of an order two subgroup,
if $2a_1=d$, then $A=\{0,d,a_1\}=\{0,2a_1,a_1\}$ is an arithmetic
progression, and if $-2a_1=d$, then
$B=\{0,d,a_1+d\}=\{0,-2a_1,-a_1\}$ is an arithmetic progression, all
contradictions as well. So we can assume w.l.o.g. $|B|\geq 4$.

\emph{\textbf{The Case $G$ Finite.}} At this point we assume $G$ is
finite, and will handle the case $G$ infinite afterwards by a
separate argument. Consequently, in view of Lemma
\ref{kemp_lem_gen2}, it follows that neither $A+B$ nor
$\overline{A+B}$ is quasi-periodic, and that
$\langle\gamma-\overline{A+B}\rangle=G$, where
$\gamma\in\overline{A+B}$. In view of Lemma \ref{kemp_lem_punc_QP},
it follows that we can assume
$\db(A,\mathcal{QP}),\,\db(B,\mathcal{QP})\geq 2$, else the proof is
complete. In view of Lemma \ref{kemp_lem_sec_transfer}, and since
$|A|,\,|B|,\,\db(A+B,\mathcal{P})\geq 3$, it follows that we can
assume
$\db(\overline{B},\mathcal{P}),\,\db(\overline{A},\mathcal{P})\geq
3$, else the proof is complete, whence in view of Lemma
\ref{kemp_lem_punc_QP} and Proposition \ref{AB_dual}, it follows
that we can assume $\db(\overline{A+B},\mathcal{QP})\geq 2$, else
$\db(B,\mathcal{QP})=1$, a contradiction.

Suppose $|N_1^b(A,B)|\geq 2$ for some $b\in B$. If $A+(B\setminus
b)$ is periodic with maximal period $H_a$, then the
non-extendibility of $B$ implies that $B\setminus b$ is
$H_a$-periodic, whence $B$ is quasi-periodic, a contradiction. Hence
$A+(B\setminus b)$ is aperiodic, whence Kneser's Theorem implies
$|N_1^b(A,B)|=2$. Thus we can apply KST to the pair $(A,B\setminus
b)$. Since $\langle A\rangle =G$ and since $A$ is
not-quasi-periodic, it follows that the quasi-period from KST must
be $G$. Hence, since $|A|,|B\setminus b|>1$, and since
$|A+(B\setminus b)|\leq |G|-2$ (in view of $\db(A+B,\mathcal{P})\geq
3$), it follows from KST that both $A$ and $B\setminus b$ are
arithmetic progressions, whence Lemma \ref{kemp_lem_AP_transfer}
completes the proof. So we can assume $|N_1^b(A,B)|\leq 1$ for all
$b\in B$. Likewise, $|N_1^a(B,A)|\leq 1$ for all $a\in A$.

If $|A|=|\overline{A+B}|=3$, then in view of Proposition
\ref{AB_dual} and the completed case $|A|=|B|=3$, it follows that
the theorem holds with type (VII). Therefore we can assume at most
one of $|A|$, $|B|$ and $|\overline{A+B}|$ is equal to three. If
$h_d(A)\leq 1$, then $\db(A,\mathcal{QP})\geq 2$ implies that
$\db(A,\mathcal{AP})\leq 1$. Likewise for $B$. Thus in view of
Lemmas \ref{kemp_lem_AP_transfer} and \ref{kemp_lem_AAP_transfer},
it follows that $h_d(A),\,h_d(B)\geq 2$ for all nonzero $d\in G$,
else the proof is complete. Likewise, if $h_d(\overline{A+B})\leq
1$, then $\db(\overline{A+B},\mathcal{QP})\geq 2$ implies that
$\db(\overline{A+B},\mathcal{AP})\leq 1$, whence in view of
Proposition \ref{AB_dual} and Lemmas \ref{kemp_lem_AP_transfer} and
\ref{kemp_lem_AAP_transfer}, it follows that either $h_d(B)\leq 1$
for some nonzero $d\in G$, or else $\db(B,\mathcal{QP})\leq 1$, both
contradictions. Therefore we can assume $h_d(\overline{A+B})\geq 2$
for all nonzero $d\in G$ as well. Consequently,
$\db(C,\mathcal{AP})\geq 2$ for $C\in\{A,B,\overline{A+B}\}$.

Suppose $|A|=3$. Hence $|B|,\,|\overline{A+B}|\geq 4$ as noted
above. If $|N_1^b(A,B)|=0$ for all $b\in B$, then Lemma
\ref{kemp_secondlemma} implies $|B|\geq |G|-6$, whence
$|\overline{A+B}|\leq 3$, a contradiction. Therefore, since
$N_1^b(A,B)\leq 1$ for all $b\in B$, it follows that we can assume
$|N_1^b(A,B)|=1$ for some $b\in B$. We proceed by induction on
$|B|$, with our domain restricted to pairs $(A,B')$ with each member
a non-quasi-periodic, generating subset.

Since $|N_1^b(A,B)|=1$, and since $A+B$ is not quasi-periodic, it
follows that $A+(B\setminus b)=(A+B)\setminus \gamma$ is aperiodic
with $\gamma\in A+B$. Suppose (\ref{Kemp_extendible}) holds for $A$
and $B\setminus b$. In view of Corollary \ref{nec-suff-cond}, it
follows that if the extended pair $A\cup\{\alpha\}$ and $B\setminus
b\cup\{\beta\}$ is periodic without a unique expression element,
then $A\cup\{\alpha\}+(B\setminus b\cup\{\beta\})=A+(B\setminus b)$.
Thus, since $A+(B\setminus b)$ is aperiodic, it follows that we can
apply KST to $A'=A\cup \{\alpha\}$ and $B'=B\setminus
b\cup\{\beta\}$. Since $\db(A,\mathcal{QP})\geq 2$, it follows that
$\db(A',\mathcal{QP})\geq 1$. Hence, since $\langle A\rangle =G$, it
follows that KST can only hold with subgroup $G$, whence either
$|A+B|\geq |A'+B'|-1\geq |G|-2$, or else $\db(A,\mathcal{AP})\leq
1$, both contradictions. Therefore we can assume
(\ref{Kemp_extendible}) does not hold. Consequently, the pair
$(A,B\setminus b)$ must be non-extendible. Hence, in view of Lemma
\ref{kemp_lem_gen2}, it follows that $B\setminus b$ is not
quasi-periodic and that $\langle B\setminus b\rangle=G$. Thus we can
apply the induction hypothesis to the pair $(A,B\setminus b)$ and
assume (\ref{Kemp_extendible}) does not hold.

Hence, since $\langle A\rangle =G$ and since $A$ is not
quasi-periodic, it follows that $H_a=G$ in Theorem
\ref{KST_Step_Beyond}. Hence, since $|A|,\,|B\setminus b|\geq 3$,
since $|\overline{A+(B\setminus b)}|\geq 4$, and since
$\db(A,\mathcal{QP})\geq 2$, it follows that we must have type (VI)
with $H_a=G$, whence $|B|=4$ and w.l.o.g. $A=\{0,d,x\}$ and
$B=\{0,d,x,b\}$ with $N_1^b(A,B)=\{x+b\}$. Since
$\db(A,\mathcal{AP})\geq 2$, then it follows that  $x\notin
\{2d,3d,-d,-2d\}$, $d\notin\{2x,3x,-x,-2x\}$ and
$0\notin\{2x-d,3x-2d,2d-x,3d-2x\}$. Additionally, $A$ not
quasi-periodic implies $A$ does not contain a coset of an order two
subgroup. Hence the previous two sentences yield that
$$\{0,d,2d\},\; \{x,x+d\},\;\{2x\}$$ are the three distinct
$d$-components of $A+A\subseteq A+B$. Since $|A+B|=7$, it follows
that $b+A$ contains exactly one element distinct from these $6$.
Since $N_1^b(A,B)=\{x+b\}$, it follows that this element must be
$x+b$. Thus the component $\{0,d\}+b$ of $b+A$ must be contained in
either $\{0,d,2d\}$ or $\{x,x+d\}$, whence $b=0,\,d$ or $x$, all
contradictions, completing the induction. So we may w.l.o.g. assume
$|A|\geq |B|\geq 4.$

Suppose $|\overline{A+B}|\leq 3$. Hence $\db(A+B,\mathcal{P})\geq 3$
implies that $|\overline{A+B}|=3$. Thus in view of Proposition
\ref{AB_dual}, the above case, and Corollary \ref{cor1}, it follows
that $\db(A,\mathcal{QP}\cup\mathcal{AP})\leq 1$, a contradiction.
So we can also assume $|\overline{A+B}|\geq 4$.

If $e+B\subseteq A$ for all $e\in A-B$, then $A-B+B=A$, implying
from Kneser's Theorem that $A$ is periodic, a contradiction. Thus we
can choose $e\in A-B$ such that $|(e+B)\cap A|$ is maximal subject
to $|(e+B)\cap A|<|B|$. Let $B(e)=(e+B)\cap A$, and let
$A(e)=(e+B)\cup A$. Note that $|A(e)|+|B(e)|=|A|+|B|$, that
$A(e)+B(e)\subseteq e+A+B$, and that $B(e)$ is non-empty. Assume by
induction that the theorem holds for non-quasi-periodic, generating
subsets $A'$ and $B'$ either with
$\min\{|A'|,|B'|\}<\min\{|A|,|B|\}=|B|$, or else with
$\min\{|A'|,|B'|\}=\min\{|A|,|B|\}=|B|$ and $|A'|+|B'|>|A|+|B|$. We
have already verified the base of the induction when $|A|+|B|\geq
|G|-3$ or when $\min\{|A|,|B|\}\leq 3$.


\emph{Case 1:} $A(e)+B(e)$ is periodic with maximal period $H_a$.

Suppose that $B(e)$ is not $H_a$-periodic. Then $B(e)$ must have an
$H_a$-hole $x$. Hence, since $A(e)+B(e)$ is $H_a$-periodic, it
follows that
$$(A\cup
(e+B))+(B(e)\cup\{x\})=A(e)+(B(e)\cup\{x\})=A(e)+B(e)\subseteq
e+A+B.$$ Consequently, $x-e+A\subseteq A+B$ and $x+B\subseteq A+B$.
Since $x$ is an $H_a$-hole in $B(e)$, then either $x\notin A$ or
$x-e\notin B$. Hence, in view of the last two sentences, it follows
that we can contradict the non-extendibility of the pair $(A,B)$ by
either adding $x$ to $A$ (if $x\notin A$) or else by adding $x-e$ to
$B$ (if $x-e\notin B$). So we may assume that $B(e)$ is
$H_a$-periodic.

Let $\rho$ be the number of $H_a$-holes contained in the pair $A(e)$
and $B(e)$, and let $\rho'$ be the number of $H_a$-holes contained
in the pair $A$ and $B$. Partition the set $A$ into the disjoint
sets $A\cap (e+B)$, $A_1$ and $A_2$, where $A_1$ consists of those
elements of $A$ which modulo $H_a$ are contained in $\phi_a(A)\cap
\phi_a(e+B)$ but which are not in $A\cap (e+B)$, and where $A_2$ are
the remaining elements of $A$ not contained modulo $H_a$ in
$\phi_a(A)\cap \phi_a(e+B)$. Likewise partition the set $e+B=(A\cap
(e+B))\cup B_1\cup B_2$. Let $\rho''$ be the number of $H_a$-holes
contained among $A_2$ and $B_2$. Since $A\cap (e+B)=B(e)$ is
$H_a$-periodic, it follows that $\phi_a(A_1)=\phi_a(B_1)$. Hence,
since $A_1\cap B_1$ is empty, it follows that $|A_1|+|B_1|=|A_1\cup
B_1|\leq |H_a||\phi_a(A_1)|$. Hence, since $A\cap (e+B)$ is
$H_a$-periodic, it follows that
\begin{equation}\label{trans-per-1}\rho=\rho''+|H_a|\cdot |\phi_a(A_1)|-|A_1|-|B_1|\geq \rho''.
\end{equation} Applying Kneser's Theorem to $(A(e),B(e))$,
it follows that
\begin{equation}\label{trans-per-2}|A(e)+B(e)|\geq
|A|+|B|-|H_a|+\rho.
\end{equation}

Suppose $b'\in B_2$ with $|\phi_a(B_{b'}+A)\setminus
\big(\phi_a(A(e)+B(e))\big)|=t\geq 1$. Thus $\rho''\geq
|H_a|-|B_{b'}|$. Hence in view of (\ref{trans-per-2}) and
(\ref{trans-per-1}) it follows that
$$|A+B|\geq |A(e)+B(e)|+t|B_{b'}|\geq
|A|+|B|-|H_a|+\rho''+|B_{b'}|\geq |A|+|B|,$$ whence equality holds.
However, equality in the above estimate implies $t=1$ and
$e+A+B=(A(e)+B(e))\cup (B_{b'}+A)$, whence $A+B$ is quasi-periodic,
a contradiction. So we can assume $\phi_a(B_{b'}+A)\subseteq
\phi_a(A(e)+B(e))$ for all $b'\in B_2$, whence the non-extendibility
of $B$ implies $B_2$ is $H_a$-periodic (or empty). By the same
argument applied to $a'\in A_2$ it follows that
$\phi_a(A_{a'}+e+B)\subseteq \phi_a(A(e)+B(e))$ for all $a'\in A_2$,
and that $A_2$ is $H_a$-periodic (or empty). Consequently
$\rho''=0$.

Let $A_1=A_{\alpha_1}\cup\ldots A_{\alpha_n}$ and
$B_2=B_{\beta_1}\cup\ldots B_{\beta_n}$ be $H_a$-coset
decompositions of $A_1$ and $B_1$, with
$\phi_a(\alpha_i)=\phi_a(\beta_i)$. In view of the result of the
previous paragraph and Lemma \ref{kemp_lem_sec_transfer}, it follows
that $n\geq 3$, else the proof is complete.

Note that $e'\in A_{\alpha_i}-B_{\beta_i}\subseteq H_a$ are exactly
those elements such that $(e'+B_{\beta_i})\cap A_{\alpha_i}$ is
nonempty. Additionally, since $A\cap (e+B)$ is $H_a$-periodic, and
since $e'\in H_a$, it follows that $A\cap (e+B)\subseteq A\cap
(e'+e+B)$. Thus, in view of the previous two sentences, unless
$e'+e+B\subseteq A$, then the element $e'+e$ will contradict the
maximality of $e$. Hence in order to avoid this contradiction we
must have: (a) $B_2$ empty (else w.l.o.g. there will be an
$H_a$-coset $\beta+H_a$ which intersects $e'+e+B$ but not $A$), and
(b) $e'+B_{\beta_i}\subseteq A_{\alpha_i}$ for each $e'\in
A_{\alpha_i}-B_{\beta_i}$ (else w.l.o.g. there will be an element
from the coset $\alpha_i+H_a$ contained in $e'+e+B$ but not in
$A'$), and (c) $A_{\alpha_i}-B_{\beta_i}=A_{\alpha_j}-B_{\beta_j}$
for all $i$ and $j$ (else w.l.o.g. there will be an element $e'\in
A_{\alpha_i}-B_{\beta_i}$ but $e'\notin A_{\alpha_j}-B_{\beta_j}$,
whence the elements from the coset $e'+\alpha_j+H_a$ contained in
$e'+e+B$ will not be contained in $A$, but some element from the
coset $e'+\alpha_i+H_a$ contained in $e'+e+B$ will be contained in
$A$).

Since $e'+B_{\beta_i}\subseteq A_{\alpha_i}$ for each $e'\in
A_{\alpha_i}-B_{\beta_i}$, it follows that
$A_{\alpha_i}-B_{\beta_i}+B_{\beta_i}=A_{\alpha_i}$, implying that
$B_{\beta_i}-B_{\beta_i}\subseteq H(A_{\alpha_i})$, where
$A_{\alpha_i}$ is maximally $H(A_{\alpha_i})$-periodic. Hence
$A_{\alpha_i}-B_{\beta_i}=-\beta_i+A_{\alpha_i}$. Thus, since
$A_{\alpha_i}-B_{\beta_i}=A_{\alpha_j}-B_{\beta_j}$ for all $i$ and
$j$, it follows that $A_{\alpha_i}=A_{\alpha_j}+(\beta_i-\beta_j)$.
Consequently, the $A_{\alpha_i}$ are all just translates of one
another, implying that $H(A_{\alpha_i})=H(A_{\alpha_j})=H_{ka}\leq
H_a$, and that $|\phi_{ka}(B_{\beta_i})|=1$ (whence
$|B_{\beta_i}|\leq |H_{ka}|$), for all $i$ and $j$. Note $H_{ka}$
must be a proper subgroup of $H_a$, else $A_{\alpha_i}\cap
B_{\beta_i}$ would be nonempty, a contradiction. Thus, since $A_2$
is $H_a$-periodic (or empty), it follows that $A$ is
$H_{ka}$-periodic, whence $|H_{ka}|=1$. Hence
$B_{\beta_i}-B_{\beta_i}\subseteq H(A_{\alpha_i})=H_{ka}$ implies
that $|B_{\beta_i}|=1$ for all $i$.

For each partially filled $H_a$-coset $F_{i}$ in $A+B$, it follows
that there must be at least one pair $A_{\alpha_i}$ and
$B_{\beta_j}$ such that $A_{\alpha_i}+B_{\beta_j}\subseteq F_i$.
Since $A+B$ is not quasi-periodic, it follows that there are at
least two distinct partially filled $H_a$-cosets in $A+B$. In view
of the non-extendibility of $A$, it follows that each $A_{\alpha_i}$
must have a $B_{\beta_\sigma(i)}$ such that
$A_{\alpha_i}+B_{\beta_\sigma(i)}\subseteq F_j$ for some $j$.
Likewise for each $B_{b_i}$. Hence in view of Proposition
\ref{matching_prop} and $n\geq 3$, it follows that there exist
distinct $i$ and $i'$ and distinct $j$ and $j'$ such that
$A_{\alpha_i}+B_{\beta_j}$ and $A_{\alpha_{i'}}+B_{\beta_{j'}}$ are
each disjoint from $A(e)+B(e)$ and $\phi_a(\alpha_i+\beta_j)\neq
\phi_a(\alpha_{i'}+\beta_{j'})$. Hence in view of
(\ref{trans-per-2}) it follows that
\begin{equation}\label{kemp-trans-5}|A+B|\geq
|A(e)+B(e)|+|A_{\alpha_i}+B_{\beta_j}|+|A_{\alpha_{i'}}+B_{\beta_{j'}}|\geq
|A|+|B|+\rho+|A_{\alpha_i}|+|A_{\alpha_{i'}}|-|H_a|,\end{equation}
with equality possible only if there are exactly two partially
filled $H_a$-cosets in $A+B$.

Since $|B_{\beta_i}|=1$ for all $i$, and since $\rho''=0$, it
follows in view of (\ref{trans-per-1}) that
$\rho=n(|H_a|-1)-\Sum{i=1}{n}|A_{\alpha_i}|\geq
2(|H_a|-1)-|A_{\alpha_{i'}}|-|A_{\alpha_i}|$. Thus
(\ref{kemp-trans-5}) implies $|A+B|\geq |A|+|B|+|H_a|-2$, whence
$|H_a|=2$ and equality must hold in (\ref{kemp-trans-5}). Hence
there are exactly two partially filled $H_a$-cosets in $A+B$. Thus,
since $|H_a|=2$ implies each partially filled $H_a$-coset contains
one hole, it follows that $\db(A+B,\mathcal{P}_{H_a})\leq 2$,
contradicting that $\db(A+B,\mathcal{P})\geq 3$, and completing the
proof. So we may assume $A(e)+B(e)$ is aperiodic, whence it follows
in view of Kneser's Theorem that either
$|A(e)+B(e)|=|A(e)|+|B(e)|-1$ or $|A(e)+B(e)|=|A(e)|+|B(e)|$. We
proceed based on $|B(e)|$.

\emph{Case 2:} $|B(e)|\geq 3$.

Suppose $|A(e)+B(e)|=|A(e)|+|B(e)|-1$. Thus we can apply KST to
$A(e)+B(e)=(e+A+B)\setminus \{\gamma\}$ with $\gamma\in e+A+B$.
Hence, since $\langle \gamma'-\overline{A+B}\rangle=G$, for
$\gamma'\in\overline{A+B}$, and since
$\db(\overline{A+B},\mathcal{QP})\geq 2$, it follows that the
quasi-period from KST must be $G$. Hence from KST it follows that
either $\db(\overline{A+B},\mathcal{AP})\leq 1$ or else $|A+B|\geq
|G|$, both contradictions. So we can assume $A(e)+B(e)=e+A+B$.

Suppose $(A(e),B(e))$ is extendible. Hence we can apply KST to
$A(e)\cup\{\alpha\}+B(e)\cup\{\beta\}=e+A+B$. As in the previous
paragraph, the quasi-period from KST must be $G$, whence either
$\db(\overline{A+B},\mathcal{AP})\leq 0$ or else $|A+B|\geq |G|-1$,
both contradictions. So we can assume $(A(e),B(e))$ is
non-extendible. Thus, since $\overline{A+B}$ is not quasi-periodic,
and since $\langle\gamma-\overline{A+B}\rangle=G$, it follows in
view of Proposition \ref{AB_dual} and Lemma \ref{kemp_lem_gen2} that
$A(e)$ and $B(e)$ are both non-quasi-periodic, generating subsets,
whence the theorem holds for $A(e)$ and $B(e)$ by induction
hypothesis. Hence in view of Corollary \ref{cor1} it follows that
$\db(\overline{A+B},\mathcal{QP}\cup\mathcal{AP})\leq 1$, a
contradiction.

\emph{Case 3:} $|B(e)|=1$.

Let $T$ be the subset of $A-B$ such that $T+B\subseteq A$. Thus we
can apply Theorem \ref{AB_large_from_A-B} with $k=1$, whence Theorem
\ref{AB_large_from_A-B}(ii) implies \be\label{T_inequality} |T|\geq
|A|\frac{|A||B|-|A|-|B|}{(|A|+|B|)(|B|-1)}.\ee If $|T|\leq 1$, then
(\ref{T_inequality}) implies
$(|B|-1)|A|^2-(2|B|-1)|A|-|B|(|B|-1)\leq 0$. This is an increasing
function of $|A|$ for $|A|>\frac{2|B|-1}{2|B|-2}$, whence $|A|\geq
|B|\geq 4$ implies $$0\geq
(|B|-1)|B|^2-(2|B|-1)|B|-|B|(|B|-1)=|B|(|B|^2-4|B|+2),$$
contradicting that $|B|\geq 4$. Therefore we can assume $|T|\geq 2$.
If $|T|\leq |A|-|B|-2$, then (\ref{T_inequality}) implies, in view
of $|B|\geq 4$, that $|A|\leq \frac{2|B|-|B|^3-|B|^2}{|B|-2}<0$, a
contradiction. Therefore we can assume \be \label{T2-bound} |T|\geq
|A|-|B|-1.\ee

Let $A'=A\setminus (T+B)$. Hence $A=A'\cup (T+B)$ and $|T+B|=
|A|-|A'|$. Let $a'\in A'$ and $a\in A$. If $(a'+B)\cap (a+B)$ is
nonempty, then $a'+b'=a+b$ for some $b,\,b'\in B$. Hence $a'\in
(a'-b)+B$ and $a=a'-b+b'\in (a'-b)+B$. Hence, if $a'$ and $a$ are
distinct, then $|(a'-b+B)\cap A|\geq 2$, whence $a'-b+B\subseteq A$.
Hence $a'-b\in T$ and $a'\in T+B$, a contradiction. Therefore every
element $a'+b$, for $a'\in A'$ and $b\in B$, is a unique expression
element in $A+B$. Hence $$|A+B|\geq |A'||B|+|(A\setminus A')+B|\geq
|A|+|A'|(|B|-1)= |A|+|B|+(|A'|-1)(|B|-1)-1.$$ Thus in view of
$|A+B|=|A|+|B|$ and $|B|\geq 3$, it follows that $|A'|\leq 1$. Thus
since $\db(A,\mathcal{QP})\geq 2$, it follows that $T+B$ is
aperiodic.

Suppose $|\nu_b(B,-B)|\geq 2$ for some nonzero $b$. Hence
$|(b+t+B)\cap (t+B)|\geq 2$ for $t\in T$. Thus, since $t+B\subseteq
A$, it follows that $|(b+t+B)\cap A|\geq 2$, whence
$(b+t+B)\subseteq A$. Consequently, $b+T+B\subseteq A$, whence the
definition of $T$ implies that $\{0,b\}+T=T$. Thus, since $b$ is
nonzero, it follows in view of Kneser's Theorem that $T$ is
periodic, contradicting that $T+B$ is aperiodic. So we can assume
$|\nu_b(B,-B)|\leq 1$ for all nonzero $b$. Consequently, $B$ is a
Sidon set and $|2B|=\frac{|B|(|B|+1)}{2}$.

Since $|A'|\leq 1$, it follows that $T+B=A\setminus \alpha$ for some
$\alpha\in G$, whence $T+2B=(A\setminus \alpha)+B$. Hence, since
$|N_1^a(B,A)|\leq 1$ for all $a\in A$ (shown in the second paragraph
of the case $G$ finite), it follows that $T+2B=(A+B)\setminus
\gamma$ for some $\gamma\in G$. Thus $A+B$ not quasi-periodic
implies that $T+2B$ is aperiodic.

If the inequality in (\ref{T2-bound}) is strict, then since $T+2B$
is aperiodic, it follows from Kneser's Theorem that $|A|+|B|\geq
|T+2B|\geq |T|+|2B|-1=|A|-|B|+\frac{|B|(|B|+1)}{2}-1$, implying
$|B|<4$, a contradiction. Therefore we can assume $|T|=|A|-|B|-1$.
Hence (\ref{T_inequality}) implies that \be\label{ticktock}|A|\geq
|B|(|B|-1)(|B|+1).\ee Note (where $M$ is as defined in Theorem
\ref{AB_large_from_A-B}) that
\be\label{ticktock2}-\frac{M}{|B|}=-|T|(|B|-1)=(|B|-|A|+1)(|B|-1)\equiv
(|B|+1)(|B|-1)\mod |A|.\ee In view of (\ref{ticktock}) it follows
that $1\leq (|B|+1)(|B|-1)\leq |A|$. Hence in view of
(\ref{ticktock2}) it follows that $x=(|B|+1)(|B|-1)|B|$ in Theorem
\ref{AB_large_from_A-B}. Hence Proposition
\ref{AB_large_from_A-B}(iii) and $|A+B|=|A|+|B|$ imply
$$|A|+|B|=|A+B|\geq
\frac{|A|^2|B|^2(|A||B|^2-|A||B|+|B|^3-|B|)}{|A||B|^2(|A||B|^2-|A||B|)}=
\frac{|A||B|-|A|+|B|^2-1}{|B|-1},$$ which implies $|B|\leq 1$, a
contradiction.

\emph{Case 4:} $|B(e)|=2$.

If $|A(e)+B(e)|=|A(e)|+|B(e)|-1$, then the arguments from the
analogous part of Case 2 complete the proof.  Therefore we can
assume $|A(e)+B(e)|=|A(e)|+|B(e)|$, $A(e)+B(e)=e+A+B$ and
$c_d(A(e))=2$, for $d$ equal to the difference of elements in
$B(e)$. Thus $A(e)+B(e)=e+A+B$ implies that
$c_d(A+B)=c_d(\overline{A+B})\leq c_d(A(e))=2$. If
$c_d(\overline{A+B})=1$, then $\overline{A+B}$ not quasi-periodic
implies $\overline{A+B}$ is an arithmetic progression, a
contradiction. Therefore $c_d(\overline{A+B})=2$. Thus, in view of
Lemma \ref{kemp_lem_2-comp-transfer}, $\db(B,\mathcal{QP}\cup
\mathcal{AP})\geq 2$, and Proposition \ref{AB_dual}, it follows that
$c_d(A),c_d(B)\leq 2$. Furthermore, we must have $c_d(A)=c_d(B)=2$
by the same reasoning used to establish this for $\overline{A+B}$.

Since $\db(A+B,\mathcal{P})\geq 3$ and $c_d(A+B)=2$, it follows that
$\db(A+B+\{0,d\},\mathcal{P})\geq 1$.  Suppose
$\db(A+B+\{0,d\},\mathcal{P})=1$. Thus by the arguments used to
establish the theorem when $\db(A+B,\mathcal{P})=1$, it follows that
(\ref{Kemp_extendible}) holds for $A$ and $B+\{0,d\}$ with
$A\cup\{\alpha\}+(B+\{0,d\})\cup\{\beta\}$ periodic with maximal
period $H_a$. If $A\cup\{\alpha\}+(B+\{0,d\})\cup\{\beta\}$ contains
no unique expression element, then Corollary \ref{nec-suff-cond}
implies that $A\cup\{\alpha\}+(B+\{0,d\})\cup\{\beta\}=A+B+\{0,d\}$.
Hence, since $\db(A+B+\{0,d\},\mathcal{P})\geq 1$, it follows that
we can assume this does not happen, whence we can apply KST to
$A\cup\{\alpha\}+(B+\{0,d\})\cup\{\beta\}$. Since
$\db(A,\mathcal{QP})\geq 2$, and since $\langle A\rangle =G$, it
follows that the quasi-period from KST must be $G$, whence
$A\cup\{\alpha\}+(B+\{0,d\})\cup\{\beta\}$ periodic and $c_d(A+B)=2$
imply that $|\overline{A+B}|\leq 3$, a contradiction. So we can
assume $\db(A+B+\{0,d\},\mathcal{P})\geq 2$.

If the pair $(A,B+\{0,d\})$ is extendible, then since
$\db(A,\mathcal{QP})\geq 2$, since $\langle A\rangle =G$, and since
$\db(A+B+\{0,d\},\mathcal{P})\geq 2$, it follows in view of KST that
$\db(A,\mathcal{AP})\leq 1$, a contradiction. Therefore we can
assume $(A,B+\{0,d\})$ is non-extendible. Thus, since $A$ is not
quasi-periodic and since $\langle A\rangle =G$, it follows in view
of Lemma \ref{kemp_Lemma_qp} that $B+\{0,d\}$ is not quasi-periodic.
Also, $\langle B+\{0,d\}\rangle \geq \langle B\rangle =G$ implies
$\langle B+\{0,d\}\rangle =G$. Thus we can apply the induction
hypothesis to the pair $(A,B+\{0,d\})$. Since $|A|,\,|B+\{0,d\}|\geq
4$, with neither $A$ nor $B+\{0,d\}$ quasi-periodic, it follows that
we cannot have type (V-VII). Since $\db(A,\mathcal{QP})\geq 2$, it
follows that we cannot have type (VIII). Thus
(\ref{Kemp_extendible}) holds for $A$ and $B+\{0,d\}$. Since
$\langle A\rangle =G$, and since $\db(A,\mathcal{QP})\geq 2$, it
follows that the quasi-period from KST must be $G$. Hence, since
$\db(A,\mathcal{AP})\geq 2$, and since $c_d(A+B)=2$, it follows in
view of KST that $|\overline{A+B}|\leq 4$, whence
$|\overline{A+B}|=4$.

Suppose $|B|\geq 5$. Hence we can apply the induction hypothesis to
$(-A,\overline{A+B})$. Thus, since $\db(A,\mathcal{QP})\geq 2$, it
follows in view of Corollary \ref{cor1} that
$\db(-A,\mathcal{AP})=\db(A,\mathcal{AP})\leq 1$, a contradiction.
So we can assume $|B|=|\overline{A+B}|=4$. Note that if the theorem
holds for $(-B,\overline{A+B})$, then in view of
$\db(A,\mathcal{QP})\geq 2$ it follows from Corollary \ref{cor1}
that $\db(-B,\mathcal{AP})=\db(B,\mathcal{AP})\leq 1$, a
contradiction. Consequently, it follows that case $G$ finite will be
complete once we complete the case with $|A|=|B|=4$ and
$c_d(A)=c_d(B)=c_d(A+B)=2$. We proceed to do so.

Suppose $|N_1^b(A,B)|>0$ for some $b\in B$. Hence in view of
$|N_1^b(A,B)|\leq 1$, it follows that $|N_1^b(A,B)|=1$. If the pair
$(A,B\setminus b)$ is extendible, then the theorem holds for the
pair $(A,B\setminus b)$. Otherwise, since $A$ is not quasi-periodic,
since $\langle A\rangle =G$, and since $A+B$ is not quasi-periodic
(implying $A+(B\setminus b)$ is aperiodic), it follows from Lemma
\ref{kemp_lem_gen2} that $B\setminus b$ is not quasi-periodic and
$\langle B\setminus b\rangle =G$, whence the theorem holds for the
pair $(A,B\setminus b)$ by induction hypothesis.  In the latter
case, Corollary \ref{cor1} implies $\db(A,\mathcal{AP}\cup
\mathcal{QP})\leq 1$, a contradiction. In the former case, we can
apply KST to $A\cup \{\alpha\}$ and $B\cup\{\beta\}$. Since
$\db(A,\mathcal{QP})\geq 2$, and since $\langle A\rangle =G$, it
follows that the quasi-period must be $G$, whence
$|\overline{A+B}|\geq 4$ likewise implies that
$\db(A,\mathcal{AP})\leq 1$, again a contradiction.  So we can
assume there are no unique expression elements in $A+B$.

Let $A_1$ and $A_2$ be the two $d$-components of $A$, and let $B_1$
and $B_2$ be the two $d$-components of $B$. Suppose for some $i$,
say $i=1$, that $|A_1|\geq 3$. Hence in view of $h_d(B)\geq 2$, it
follows that $|A_1+B|=|B|+4=|B|+|A|$. Thus $A_1+B=A+B$. Since
$c_d(A+B)=2$, this implies that $A_1+B_1$ and $A_1+B_2$ are distinct
components in $A+B$, and that each of the four end terms of
components in $A_1+B$ is a unique expression element. We may assume
w.l.o.g. that $|B_1|\geq |B_2|$. Thus the component $A_1+B_1$ is
longer than either of the components $A_2+B_1$ and $A_2+B_2$. Hence,
since $A+B$ contains no unique expression element, and since
$A_1+B=A+B$, it follows that the only way the two end terms of
$A_1+B_1$ will not be unique expression elements in $A+B$ is if
$A_2+B\subseteq A_1+B_1$. Thus both the end terms of $A_1+B_2$ are
unique expression elements, a contradiction. So we can assume
$|A_i|=2$ for $i=1,2$. Likewise $|B_i|=2$ for $i=1,2$.

Since $h_d(B)\geq 2$, it follows that $A_1+B_1$ and $A_1+B_2$ are
distinct components in $A_1+B$, and that all four of the end terms
are unique expression elements in $A_1+B$. Since
$|A_i+B_j|=|A_{i'}+B_{j'}|$ for all $i,i',j,j'$, and since
$c_d(A+B)=2$, it follows that only way that these four terms can all
not be unique expression elements in $A+B$ is if $A_2+B_1=A_1+B_2$
and $A_2+B_2=A_1+B_1$. Thus $A_1+B=A_2+B$, implying
$|A+B|=|B|+2<|A|+|B|$, a contradiction. Consequently, we conclude
that Theorem \ref{KST_Step_Beyond} holds for $G$ finite.

\emph{\textbf{The Case $G$ Infinite.}} Assume $G$ is infinite. Since
$A$ and $B$ are finite, we may w.l.o.g. assume $G$ is finitely
generated. Hence $G\cong \Z^l\times T$, where $T$ is the torsion
subgroup of $G$. By translation, we can assume all non-torsion
coordinates for all $a\in A$ and $b\in B$ are non-negative. Let $M$
be the maximum integer that occurs in a non-torsion coordinate of
the $a\in A$ and $b\in B$. Let $p$ be a prime such that
$p>4(M+|T|+|A||B|)$. Let $\varphi:G\cong \Z^l\times T\rightarrow
(\Z/p\Z)^l\times T$ be the map defined by reducing all non-torsion
coordinates modulo $p$. Since $p>4(M+|T|+|A||B|)\geq 2M$, it follows
that $\varphi$ is a Freiman isomorphism of $(A+T,B+T)$ (see the
definition given before the start of the proof of Theorem
\ref{KST_Step_Beyond}), and thus also of $(A,B)$. Hence
$|\varphi(A+B)|=|\varphi(A)|+|\varphi(B)|=|A|+|B|$. Thus, since
$\varphi(G)$ is finite, it follows that we can apply Theorem
\ref{KST_Step_Beyond} to $\varphi(A)$ and $\varphi(B)$.

If $\db(\varphi(A+B),\mathcal{P}_{H_a})\leq 2$, for some nontrivial
subgroup $H_a\leq \varphi(G)$, then $$|A|+|B|=|\varphi(A+B)|\geq
|H_a|-2.$$ Hence $p>4(M+|T|+|A||B|)\geq 4(|A||B|+1)$ implies
$H_a\leq T$ (since any element outside $T$ has a coordinate with
order at least $p$, and thus is itself of order at least $p$). Hence
from the definition of $\varphi$ it follows that
$\db(A+B,\mathcal{P}_{H_a})\leq 2$, a contradiction. Therefore we
can assume $\db(\varphi(A+B),\mathcal{P})\geq 3$. Since $\langle
A\rangle =G$, it follows that $\langle\varphi(A)\rangle
=\varphi(G)$. Likewise $\langle \varphi(B)\rangle =\varphi(G)$.

Suppose that $\db(\varphi(A),\mathcal{QP}_{H_a})=0$ for some
nontrivial subgroup $H_a\leq \varphi(G)$. Hence
$|\varphi(A)|=|A|\geq |H_a|$. Thus, since $p>4(M+|T|+|A||B|)\geq
4(|A||B|+1)$, it follows that $H_a\leq T\leq G$. Hence
$\db(\varphi(A),\mathcal{QP}_{H_a})=0$ implies (in view of the
definition of $\varphi$) that $A$ is quasi-periodic with
quasi-period $H_a\leq G$, a contradiction. So we can assume
$\db(\varphi(A),\mathcal{QP})\geq 1$. Likewise
$\db(\varphi(B),\mathcal{QP})\geq 1$.

Suppose
\be\label{hihoooo}\db(\varphi(A),\mathcal{AP}_d),\,\db(\varphi(B),\mathcal{AP}_d)\leq
1,\ee for some nonzero $d\in \varphi(G)$. Consequently,
$\varphi(G)=\langle\varphi(A)\rangle$ is cyclic, implying $G\cong
\Z\times T$ with $T$ cyclic. For $x\in \varphi(G)$, let
$\overline{x}$ denote the least non-negative integer representative
of the integer coordinate of $x$.

If $d\in T$, then it follows that $\varphi(A)\subseteq T$, whence
$A\subseteq T$, contradicting that $\langle A\rangle =G$ is
infinite. Therefore we can assume $d\notin T$. Hence, by considering
$-d$ if needed, it follows that $1\leq \overline{d}\leq
\frac{p-1}{2}$. Since $\db(\varphi(A),\mathcal{AP}_d)\leq 1$, let
$P=\{p_0,p_0+d,\ldots, p_0+|A|d\}$ be an arithmetic progression with
difference $d$ that contains $\varphi(A)$.

Suppose $\overline{p_0+id}=\overline{p_0}+i\overline{d}$ for all $i$
does not hold. Hence, if $d\leq M$, then $p>4(M+|T|+|A||B|)>4M$
implies that $\{\overline{p_0},\overline{p_0+d},\ldots,
\overline{p_0+|A|d}\}$ must contain at least two elements from the
interval $(M,p)$. Hence $P$ contains at least two elements from
$\overline{A}$, contradicting the $|P\setminus A|\leq 1$. Otherwise,
$d\leq \frac{p-1}{2}$ and $M\leq \frac{p-1}{2}$ imply that
$\{\overline{p_0},\overline{p_0+d},\ldots, \overline{p_0+|A|d}\}$
contains at least $|A|-1$ elements from the interval $(M,p)$, whence
$|P\setminus A|\leq 1$ implies $|A|\leq 2$, a contradiction. So we
may assume $\overline{p_0+id}=\overline{p_0}+i\overline{d}$ for all
$i$.

Hence $A$ is contained in an arithmetic progression of difference
$(\overline{d},t)$ and at most one hole, where $t$ is the torsion
coordinate of $d$. By the same argument applied to $B$, it follows
that $B$ is also contained in an arithmetic progression of
difference $(\overline{d},t)$ with at most one hole. Hence letting
$\alpha$ be the hole in $A$, and letting $\beta$ be the hole in $B$,
it follows in view of $|A+B|=|A|+|B|$ that (\ref{Kemp_extendible})
holds, completing the proof. So we may assume that (\ref{hihoooo})
does not hold.

Suppose that $(\varphi(A),\varphi(B))$ is extendible. Hence w.l.o.g.
there exists $\alpha\in \overline{\varphi(A)}$ such that
$\varphi(A)\cup \{\alpha\}+\varphi(B)=\varphi(A)+\varphi(B)$. Thus
we can apply KST to $(\varphi(A)\cup \{\alpha\},\varphi(B))$. Since
$\db(\varphi(B),\mathcal{QP})\geq 1$, and since $\langle
\varphi(B)\rangle=\varphi(G)$, it follows that the quasi-period from
KST must be $\varphi(G)$. Hence, since
$\db(\varphi(A+B),\mathcal{P})\geq 3$, it follows from KST that
(\ref{hihoooo}) holds, a contradiction. So we can assume
$(\varphi(A),\varphi(B))$ is non-extendible.

Since $(\varphi(A), \varphi(B))$ is non-extendible, since
$|\overline{\varphi(A+B)}|,\,|\varphi(B)|\geq 4$, and since
(\ref{hihoooo}) does not hold, it follows in view of Corollary
\ref{cor1} that $\db(\varphi(A),\mathcal{QP}_{H_a})=1$ for some
nontrivial subgroup $H_a\leq \varphi(G)$. Hence
$|\varphi(A)|=|A|\geq |H_a|-1$. Thus, since $p>4(M+|T|+|A||B|)\geq
4(|A||B|+1)$, it follows that $H_a\leq T\leq G$. Observe that we
have verified all the hypotheses needed to apply Lemma
\ref{kemp_lem_punc_QP} to $(\varphi(A),\varphi(B))$. Hence Lemma
\ref{kemp_lem_punc_QP} implies that (\ref{Kemp_extendible}) holds
for $(\varphi(A), \varphi(B))$. Thus, since $\varphi$ is a Freiman
isomorphism for $(A+T,B+T)$, since $H_a\leq T$, and since the proof
of Lemma \ref{kemp_lem_punc_QP} shows $\alpha\in \varphi(A)+H_a$ and
$\beta\in \varphi(B)+H_a$, it follows that (\ref{Kemp_extendible})
holds for $(A,B)$, completing the proof.
\end{proof}

\textbf{Acknowledgements:} I wish to thank Oriol Serra for some
helpful comments regarding the manuscript.

\begin{small}

\end{small}

\end{document}